\documentclass[onefignum,onetabnum]{siamart171218}

\ifpdf
\hypersetup{
  pdftitle={A new class of high-order methods for multirate differential equations},
  pdfauthor={V.~T.~Luan, R.~Chinomona, and D.~R.~Reynolds}
}
\fi
\usepackage{lipsum}
\usepackage{epstopdf}
\usepackage{graphicx}
\usepackage{array}
\usepackage{bigstrut}
\usepackage{epsfig}
\usepackage{amssymb}
\usepackage{amsmath} 		
\usepackage{amssymb} 		
\usepackage{amstext}
\usepackage{amsfonts}		
\usepackage{url}		    
\usepackage{hhline}			
\usepackage{array}
\usepackage{tabularx}
\usepackage{makecell}
\usepackage{algorithm}
\usepackage{algorithmic}
\usepackage{enumitem}
\usepackage{color}
\usepackage{caption}
\usepackage{sub caption}
\usepackage{enumitem}
\usepackage{relsize}

\numberwithin{equation}{section}
\newcommand{\q}{\quad}
\newcommand{\ee}{{\rm e}\hspace{1pt}}
\newcommand{\dd}{\hspace{0.5pt}{\rm d}\hspace{0.5pt}}
\newcommand{\merkthree}{$\mathtt{MERK3}$}
\newcommand{\merkthrees}{$\mathtt{MERK3}^*$}
\newcommand{\merkfour}{$\mathtt{MERK4}$}
\newcommand{\merkfours}{$\mathtt{MERK4}^*$}
\newcommand{\merkfive}{$\mathtt{MERK5}$}
\newcommand{\merkfives}{$\mathtt{MERK5}^*$}
\newcommand{\merbthree}{$\mathtt{MERB3}$}
\newcommand{\merbfour}{$\mathtt{MERB4}$}
\newcommand{\merbfive}{$\mathtt{MERB5}$}
\newcommand{\merbsix}{$\mathtt{MERB6}$}
\newcommand{\mrigarkthree}{$\mathtt{MRI}$-$\mathtt{GARK}$-$\mathtt{ERK33a}$}
\newcommand{\mrigarkfour}{$\mathtt{MRI}$-$\mathtt{GARK}$-$\mathtt{ERK45a}$}
\newcommand{\mrigarkth}{$\mathtt{MRI}$-$\mathtt{GARK33a}$}
\newcommand{\mrigarkths}{$\mathtt{MRI}$-$\mathtt{GARK33a}^*$}
\newcommand{\mrigarkf}{$\mathtt{MRI}$-$\mathtt{GARK45a}$}
\newcommand{\mrigarkfs}{$\mathtt{MRI}$-$\mathtt{GARK45a}^*$}

\newcommand{\nnltermn}{\hat{N}_n(t_n,\hat{u}_n)}
\newcommand{\vn}{\hat{V}_n}

\newcommand{\Dntwo}{\widehat{D}_{n2}}

\ifpdf
  \DeclareGraphicsExtensions{.eps,.pdf,.png,.jpg}
\else
  \DeclareGraphicsExtensions{.eps}
\fi

\newsiamremark{remark}{Remark}
\newsiamremark{hypothesis}{Hypothesis}
\crefname{hypothesis}{Hypothesis}{Hypotheses}
\newsiamthm{claim}{Claim}

\headers{Multirate Exponential Rosenbrock Methods}{V.~T.~Luan, R.~Chinomona, and D.~R.~Reynolds}

\title{Multirate Exponential Rosenbrock Methods \thanks{Submitted to the editors DATE.
\funding{The first author is supported by NSF grant DMS--2012022. The second and third authors were supported in part by the U.S. Department of Energy, Office of Science, Office of Advanced Scientific Computing Research, Scientific Discovery through Advanced Computing (SciDAC) Program through the FASTMath Institute, under Lawrence Livermore National Laboratory subcontract B626484 and DOE award DE-SC0021354.}}}

\author{Vu Thai Luan\thanks{Department of Mathematics and Statistics, Mississippi State University,
 Mississippi State, MS, 39762 (\email{luan@math.msstate.edu})}
\and Rujeko Chinomona\thanks{Department of Mathematics, Southern Methodist
    University, Dallas, TX 75275-0156 (\email{rchinomona@smu.edu}, \email{reynolds@smu.edu}).}
\and Daniel R. Reynolds\footnotemark[3]
}

\usepackage{amsopn}

\begin{document}
\maketitle



\begin{abstract}
In this paper we propose a novel class of methods for high order accurate integration of multirate systems of ordinary differential equation initial-value problems.  The proposed methods construct multirate schemes by approximating the action of matrix $\varphi$-functions within explicit exponential Rosenbrock (ExpRB) methods, thereby called \emph{Multirate Exponential Rosenbrock} (MERB) methods. They consist of the solution to a sequence of modified ``fast'' initial-value problems, that may themselves be approximated through subcycling any desired IVP solver.  In addition to proving how to construct MERB methods from certain classes of ExpRB methods, we provide rigorous convergence analysis of these methods and derive efficient MERB schemes of orders two through six (the highest order ever constructed infinitesimal multirate methods).  We then present numerical simulations to confirm these theoretical convergence rates, and to compare the efficiency of MERB methods against other recently-introduced high order multirate methods.
\end{abstract}

\section{Introduction}

In this paper, we consider numerical methods to perform highly accurate time integration for multirate systems of ordinary differential equation (ODE) initial-value problems (IVPs).  The primary characteristic of these problems is that they are comprised of two or more components that on their own would evolve on significantly different time scales.  Such problems may be written in the general additive form
\begin{equation} \label{eq1.1}
  u'(t)= F(t, u(t)) := F_{f}(t,u)+F_{s}(t,u), \q t\in[t_0, T], \q u(t_0)= u_0,
\end{equation}
where $F_{f}$ and $F_{s}$ contain the ``fast'' and ``slow'' operators or variables, respectively.  Typically, either due to stability or accuracy limitations the fast processes must be evolved with small step sizes; however the slow processes could allow much larger time steps.  Such problems frequently arise in the simulation of ``multiphysics'' systems, wherein separate models are combined together to simulate complex physical phenomena \cite{Keyes2013}.   While such problems may be treated using explicit, implicit, or mixed implicit-explicit time integration methods that evolve the full problem using a shared time step size, this treatment may prove inefficient, inaccurate or unstable, depending on which time scale is used to dictate this shared step size.  Historically, scientific simulations have treated such problems using \emph{ad hoc} operator splitting schemes where faster components are ``subcycled'' using smaller time steps than slower components.  Schemes in this category include Lie--Trotter \cite{mclachlanSplittingMethods2002} and Strang--Marchuk \cite{Marchuk1968,strangConstructionComparisonDifference1968a} techniques, that are first and second-order accurate, respectively.  In recent years, however, methods with increasingly high orders of accuracy have been introduced.  Our particular interest lies in methods allowing so-called ``infinitesimal'' formulations, wherein the fast time scale is assumed to be solved exactly, typically through evolution of a sequence of modified fast IVPs,
\[
   v'(\tau) = F_{f}(\tau,v) + g(\tau),\quad \tau\in[\tau_0,\tau_f], \quad v(\tau_0)=v_0,
\]
and where the forcing function $g(\tau)$, time interval $[\tau_0,\tau_f]$, and initial condition $v_0$ are determined by the multirate method to incorporate information from the slow time scale.  In practice, however, these fast IVPs are solved using any viable numerical method, typically with smaller step size than is used for the slow dynamics.  While both the legacy Lie--Trotter and Strang--Marchuk schemes satisfy this description, each uses $g(\tau)=0$, and only couple the time scales through the initial condition $v_0$.  The first higher-order infinitesimal multirate methods were the \emph{multirate infinitesimal step} (MIS) methods \cite{schlegelMultirateRungeKutta2009,wenschMultirateInfinitesimalStep2009}, that allowed up to third order accuracy.  These have been extended by numerous authors in recent years to support fourth and fifth orders of accuracy, as well as implicit or even mixed implicit-explicit treatment of the slow time scale \cite{Bauer2019Extended,Chinomona2020,LCR2020,Sandu2019,sextonRelaxedMultirateInfinitesimal2018}.

Most higher-order ($\ge 3$) infinitesimal methods, including MIS, relaxed MIS \cite{sextonRelaxedMultirateInfinitesimal2018}, extended MIS \cite{Bauer2019Extended}, multirate infinitesimal GARK \cite{robertsCoupledMultirateInfinitesimal2020,Sandu2019}, and implicit-explicit multirate infinitesimal GARK \cite{Chinomona2020}, place no restrictions on the operators $F_{f}$ and $F_{s}$.  The corresponding order conditions for these methods are rooted in partitioned Runge--Kutta theory, to the end that the number of order conditions grows exponentially with the desired order of accuracy, to the effect that none of these methods have been proposed with order of accuracy greater than four.

In previous work, we presented an alternate approach for deriving infinitesimal multirate methods that was based on exponential Runge--Kutta (ExpRK) theory, named \emph{multirate exponential Runge--Kutta} (MERK) methods \cite{LCR2020}.  A particular benefit of this theory is that exponential Runge--Kutta methods require fewer order conditions than partitioned Runge--Kutta methods; however, to leverage this theory, MERK methods require that the fast time scale operator is autonomous and that it depends \emph{linearly} on the solution $u$, i.e., these consider the IVP
\begin{equation}
    \label{eq:semilinearIVP}
    u'(t) = F(t,u(t)) := \mathcal{L} u + \mathcal{N}(t,u), \q t\in[t_0,T], \q u(t_0)= u_0,
\end{equation}
where the ``fast'' and ``slow'' components are $F_{f}(t,u) = \mathcal{L}u$ and $F_s(t,u) = \mathcal{N}(t,u)$, respectively.  With this restriction in place, however, MERK methods have been proposed with orders of accuracy up to five.

In this work, we address the case of a non-autonomous and nonlinear fast time scale operator $F_{f}(t,u)$ by proposing to use a dynamic linearization approach that updates the operators $\mathcal{L}$ and $\mathcal{N}$ within each time step.  We then leverage this dynamic linearization approach through building multirate schemes from exponential Rosenbrock (ExpRB) methods. This new class of multirate schemes, called \emph{Multirate Exponential Rosenbrock} (MERB) methods, approximates the action of matrix $\varphi$-functions within explicit ExpRB methods, and consist of solving a sequence of modified linear ODE-IVPs, which can be integrated using any desired ODE solvers.
Moreover, we establish an elegant convergence theory for MERB methods, allowing us to determine a minimum order of accuracy for the numerical methods needed for solving the corresponding fast time scale IVPs.  In addition to this theory, we generalize the coefficients for a number of high-order ExpRB methods
and exploit their parallel stage structure to derive efficient multirate methods of very high order (including the first-ever infinitesimal multirate method of order six), with optimized numbers of modified fast IVPs.
Our numerical experiments show that these new proposed MERB schemes are uniformly the most efficient when considering slow function calls (this is particular of interest for multirate systems where the fast component is much less costly to compute than the slow component), and thus are very competitive in comparison with recently developed high order multirate methods such as MERK and MRI-GARK.

The remainder of this paper is organized as follows.  We first present the structure of ExpRB methods (Section \ref{sec2.1}).  Then in Section \ref{sec2.2} we interpret the corresponding ExpRB internal stages and time step approximations as exact solutions to modified ``fast'' initial-value problems, thereby deriving MERB methods.  In Section \ref{sec:analysis} we present rigorous convergence analysis for this family of newly-proposed methods.  Then in Section \ref{section:2.4} we construct specific multirate methods from this family for practical use,
and discuss techniques for their numerical implementation in Section \ref{section:2.5}.  In Section \ref{sec:numerical_results} we provide detailed numerical results to compare the performance of the proposed methods with the recent MERK methods of orders three through five, as well as with third and fourth order explicit MRI-GARK methods.  Finally, we provide concluding remarks and discuss avenues for future research in Section \ref{sec:conclusion}.

\section{Multirate Exponential Rosenbrock Methods}
\label{section3}


\subsection{Exponential Rosenbrock schemes}
\label{sec2.1}
ExpRB methods are constructed by linearizing the vector field $F(t,u)$ at each step along the numerical solution $(t_n,u_n)$,
\begin{equation} \label{eq2.1}
  u'(t)=F(t,u(t))=J_{n} u(t)+ V_n t+ N_n(t,u(t))
\end{equation}
with
\begin{equation}  \label{eq2.2}
\begin{aligned}
  J_n =\frac{\partial F} {\partial u}(t_n, u_n), \q V_n =\frac{\partial F} {\partial t}(t_n, u_n), \q N_{n} (t, u)= F(t,u)- J_{n}u - V_{n}t.
\end{aligned}
\end{equation}
We note that if \eqref{eq1.1} is in fact autonomous, i.e., $u'(t)=F(u(t))$, then this linearization simplifies since $V_n =0$ and $N_{n} (t, u) = N_{n} (u)= F(u)- J_{n}u$.

One can represent the exact solution to \eqref{eq2.1} at time $t_{n+1}=t_n +H$ as in \cite{Luan2014} by applying the variation-of-constants formula (a.k.a., Duhamel's principle),
\begin{equation} \label{eq2.3}
\begin{aligned}
  u(t_{n+1})= \ee^{H J_n}u(t_n) &+ \int_{0}^{H} \ee^{(H-\tau) J_n} \Big( V_{n}(t_n +\tau )+ N_n ( t_n+\tau, u(t_n+\tau )) \Big) \dd\tau\\
  = \ee^{H J_n}u(t_n) &+H \varphi_1 (H J_n)V_n t_n+H^2 \varphi_2 (H J_n)V_n \\
  &+ \int_{0}^{H} \ee^{(H-\tau) J_n} N_n( t_n+\tau, u(t_n+\tau )) \dd\tau,
\end{aligned}
\end{equation}
where $\varphi_k(Z)$ ($Z=H \mathcal{L}$) belong to the family of $\varphi$-functions given by
\begin{equation} \label{eq2.4}
  \varphi_{k}(Z)=\frac{1}{H^k}\int_{0}^{H} \ee^{(H-\tau )\frac{Z}{H}} \frac{\tau^{k-1}}{(k-1)!} \dd \tau , \quad k\geq 1.
\end{equation}
Explicit ExpRB methods approximate the integral in \eqref{eq2.3} by using a quadrature rule with nodes $c_i$ in $[0,1]$ ($i = 1,\ldots,s$)  ($c_1=0$).  Denoting the resulting approximations $u_n \approx u(t_n)$ and $U_{ni}\approx u(t_n +c_i H)$, ExpRB methods may be written as
\begin{equation} \label{eq2.6}
\begin{aligned}
  U_{ni}= u_n & + c_i H \varphi _{1} ( c_i H J_n)F(t_n, u_n) +  c^{2}_{i} H^2 \varphi _{2} ( c_i H J_n)V_n
  + H \sum_{j=2}^{i-1}a_{ij}(H J_n) D_{nj},  \\
  u_{n+1}= u_n & + H \varphi _{1} ( H J_n)F(t_n, u_n) + H^2 \varphi _{2} ( H J_n)V_n + H \sum_{i=2}^{s}b_{i}(H J_n) D_{ni},
\end{aligned}
\end{equation}
where
\begin{equation} \label{eq:Dni}
  D_{ni}=N_n (t_n+c_i H, U_{ni})- N_n (t_n, u_n ),
\end{equation}
 ($i=2,\ldots,s$) and where $D_{n1}=0$  \cite{HOS09,Luan2014}.  Here, the weights $a_{ij}(H J_n)$  and $b_i(H J_n)$ are usually chosen (by construction) as linear combinations of the $\varphi_k (c_i H J_n)$ and $\varphi_k (H J_n)$ functions given in \eqref{eq2.4}, respectively.  These unknown functions can be determined by solving  order conditions, depending on the required order of accuracy.

\begin{remark}\label{remark2.1}
(Order conditions)
For later use, in Table \ref{tb2.1} we recall the stiff order conditions for ExpRB methods up to order 6 from \cite{LO13}.  We note that an ExpRB method of order 6 only requires 7 conditions, which is much less than the 36 conditions needed for explicit Runge--Kutta or exponential Runge--Kutta methods of the same order. This is the advantage of the dynamic linearization approach \eqref{eq2.1}, and can be understood by observing from \eqref{eq2.2} that
\begin{equation} \label{eq2.7}
  \frac{\partial N_n}{\partial u}(t_n,u_n) =0 \q\text{and}\q \frac{\partial N_n}{\partial t}(t_n,u_n)= 0.
\end{equation}
This property significantly simplifies the number of order conditions, particularly for higher-order schemes. A further consequence of \eqref{eq2.7} is that from \eqref{eq:Dni} we have $D_{ni}=\mathcal{O}(H^2)$, meaning that ExpRB methods are at least of order 2.
\renewcommand{\arraystretch}{1.6}%
\begin{table}[h]
\caption{Stiff order conditions for ExpRB methods up to order 6 (from \cite{LO13}). Here $Z, K$, and $M$ denote arbitrary square matrices.}
\begin{center}
\begin{tabular}{ |c|c|c| }
\hline
No. & Order condition & Order \\
\hline
1&$\sum_{i=2}^{s} b_i (Z)c^2_i=2\varphi_3 (Z) $&3 \\
 \hline
2&$\sum_{i=2}^{s} b_i (Z)c^3_i=6\varphi_4 (Z) $&4 \\
\hline
3&$\sum_{i=2}^{s} b_i (Z)c^4_i=24\varphi_5 (Z) $&5 \\
4&$\sum_{i=2}^{s} b_i (Z)c_i  K\big( \sum_{k=2}^{i-1}a_{ik}(Z)\frac{c^2_k}{2!}-c^{3}_i \varphi_{3} ( c_i Z)\big)=0$&5 \\
\hline
5&$\sum_{i=2}^{s}b_{i}(Z)c_i^5 = 120\varphi_6(Z)$ & 6\\
6&$\sum_{i=2}^{s}b_{i}(Z) c_i^2 M \big( \sum_{k=2}^{i-1}a_{ik}(Z)\frac{c^2_k}{2!}-c^{3}_i \varphi_{3} ( c_i Z)\big)=0$& 6 \\
7&$\sum_{i=2}^{s}b_{i}(Z) c_i K\big( \sum_{k=2}^{i-1}a_{ik}(Z)\frac{c^3_k}{3!}-c^{4}_i \varphi_{4} ( c_i Z)\big)=0 $ & 6\\
\hline
\end{tabular}
\end{center}
\label{tb2.1}
\end{table}%

\end{remark}

\subsection{A multirate procedure for ExpRB methods}
\label{sec2.2}
Inspired by our recent work \cite{LCR2020}, we now show how ExpRB schemes can be interpreted as a class of multirate infinitesimal step-type methods. Namely, we construct modified differential equations whose exact solutions corresponding to the ExpRB internal stages $U_{ni}$ ($i=2,\ldots,s$) and the final stage $u_{n+1}$.
\begin{lemma}\label{Lemma2.1}
  Consider an explicit exponential Rosenbrock scheme \eqref{eq2.6} where the weights $a_{ij}(H J_n)$ and $b_{i}(H J_n)$ can be written as linear combinations of  $\varphi_k$ functions,
  \begin{equation} \label{eq2.8}
    a_{ij}(H J_n)=\sum_{k=1}^{\ell_{ij}}\alpha^{(k)}_{ij}\varphi_{k}(c_i HJ_n), \q
    b_{i}(H J_n)=\sum_{k=1}^{m_i}\beta^{(k)}_{i}\varphi_{k}(H J_n),
  \end{equation}
  and where $\ell_{ij}$ and $m_i$ are some positive integers. Then, $U_{ni}$ and $u_{n+1}$  are the
  exact solutions of the following (linear) modified differential equations 
  \begin{subequations} \label{eq2.9}
    \begin{align}
     v'_{ni}(\tau)&=J_nv_{ni}(\tau) +   p_{ni}(\tau), && v_{ni}(0) = u_n, \qquad  i=2,\ldots,s, \label{eq2.9a} \\
     v'_{n+1}(\tau)&=J_n v_{n+1}(\tau)  +   q_{n}(\tau),  &&  v_{n+1}(0) = u_n \hspace{2.3cm}  \label{eq2.9b}
    \end{align}
  \end{subequations}
  at the times $\tau= c_i H$ and $\tau= H$, respectively. Here
  $p_{ni}(\tau)$ and $q_{n}(\tau)$ are polynomials in $\tau$ given by
  \begin{subequations} \label{eq2.10}
    \begin{align}
      p_{ni}(\tau)&=N_n(t_n, u_n)+ (t_n + \tau)V_n + \sum_{j=2}^{i-1} \Big(\sum_{k=1}^{\ell_{ij}}\dfrac{\alpha^{(k)}_{ij}}{c^k_i H^{k-1} (k-1)!}\tau^{k-1}\Big) D_{nj}, \label{eq2.10a} \\
      q_{n}(\tau) &=N_n(t_n, u_n)+ (t_n + \tau)V_n +  \sum_{i=2}^{s} \Big(\sum_{k=1}^{m_i}\dfrac{\beta^{(k)}_{i} }{H^{k-1}(k-1)!}\tau^{k-1} \Big) D_{ni}.  \label{eq2.10b}
    \end{align}
  \end{subequations}
\end{lemma}
\begin{proof}
  The proof can be carried out in a very similar manner as in \cite[Theorem~3.1]{LCR2020}. Here, we only sketch the main idea. First, we insert the $\varphi_k$ functions from \eqref{eq2.4} into \eqref{eq2.8} to get the integral representations of $a_{ij}(H J_n)$ and $b_{i}(H J_n)$:
 \begin{subequations}\label{eq:coeff}
 \begin{align}
   a_{ij}(H J_n) &=\int_{0}^{c_i H} \ee^{(c_i H-\tau)J_n} \sum_{k=1}^{\ell_{ij}}\dfrac{\alpha^{(k)}_{ij}}{(c_i H)^{k} (k-1)!}\tau^{k-1}\dd\tau, \label{eq:coeff_a} \\
    b_{i}(H J_n) &=\int_{0}^{H} \ee^{(H-\tau)J_n } \sum_{k=1}^{m_i}\dfrac{\beta^{(k)}_{i}}{H^{k}(k-1)!}\tau^{k-1}\dd\tau. \label{eq:coeff_b}
 \end{align}
 \end{subequations}
 Inserting these into \eqref{eq2.6} shows that the ExpRB stages and time step update may be written as
  \begin{subequations} \label{eq2.11}
    \begin{align}
      U_{ni}&= \ee^{c_i  H J_n} u_n +\int_{0}^{c_i H} \ee^{(c_i H-\tau)J_n}   p_{ni}(\tau) \dd\tau, \q  i=2,\ldots,s,  \label{eq2.11a} \\
      u_{n+1} &= \ee^{H J_n}u_n +  \int_{0}^{H} \ee^{(H -\tau)J_n}  q_{n}(\tau) \dd\tau,  \label{eq2.11b}
    \end{align}
  \end{subequations}
which clearly show that   $U_{ni}=v_{ni}(c_i H)$ and $u_{n+1}=v_{n+1}(H)$ by means of the variation-of-constants formula applied to \eqref{eq2.9a} and \eqref{eq2.9b}, respectively.
\end{proof}
\textbf{MERB methods}.
Starting from the initial value  $u_0=u(t_0)$, equations \eqref{eq2.9} from Lemma~\ref{Lemma2.1} suggest a multirate procedure to approximate the numerical solutions $u_{n+1}$ ($n=0, 1, 2, \ldots$) obtained by ExpRB methods. Specifically, one may integrate the slow process $(V_n t + N_n(t,u))$ using a macro time step $H$, and integrate the fast process $(J_n u)$ using a micro time step $h=H/m$ (where $m>1$ is an integer representing the time scale separation factor) via solving the ``fast'' ODEs \eqref{eq2.9a} on $[0, c_i H]$ and \eqref{eq2.9b} on $[0, H]$.
Let us denote the corresponding numerical solutions of these ODEs as $\widehat{U}_{ni}$ ($ \approx v_{ni}(c_i H)=U_{ni}$) and $\hat{u}_{n+1}$ ($\approx v_{n+1}(H)=u_{n+1}$).  Then this multirate procedure consists in each step of solving \eqref{eq2.9}--\eqref{eq2.10} with the initial value $\hat{u}_n$ ($\hat{u}_0 = u_0$). Since we must linearize each step around the approximate solution $\hat{u}_n$ instead of the true value $u_n$, we denote the approximations of $J_n, V_n, N_{n} (t, u)$, and $D_{nj}$ appearing in polynomials \eqref{eq2.10} as
\begin{subequations} \label{eq:terms}
\begin{align}
\hat{J}_n &=\frac{\partial F} {\partial u}(t_n, \hat{u}_n), \ \hat{V}_n =\frac{\partial F} {\partial t}(t_n, \hat{u}_n), \  \hat{N}_n (t, u) = F(t,u)- \hat{J}_{n}u - \hat{V}_{n}t, \label{eq:JVN_hat}  \\
 \widehat{D}_{nj} &= \hat{N}_n (t_n+c_j H, \widehat{U}_{nj}) - \hat{N}_n (t_n, \hat{u}_n ).   \label{eq:D_hat}
 \end{align}
\end{subequations}
Thus, starting with $\hat{u}_0 = u_0$, for each time step $t_n\to t_{n+1}$ we solve perturbed linear ODEs for $i=2,\ldots,s$:
\begin{equation} \label{eq:MERBa}
  y'_{ni}(\tau)= \hat{J}_n y_{ni}(\tau) +   \hat{p}_{ni}(\tau), \q \tau\in [0,c_iH], \q y_{ni}(0)= \hat{u}_n,
\end{equation}
with
\begin{equation} \label{eq:MERB:poly:p}
 \hat{p}_{ni}(\tau) = \hat{N}_n (t_n, \hat{u}_n)+ (t_n + \tau) \hat{V}_n + \sum_{j=2}^{i-1} \Big(\sum_{k=1}^{\ell_{ij}}\dfrac{\alpha^{(k)}_{ij}}{c^k_i H^{k-1} (k-1)!}\tau^{k-1}\Big) \widehat{D}_{nj},
\end{equation}
to obtain
\[
  \widehat{U}_{ni}\approx y_{ni}(c_i H) \approx v_{ni}(c_i H)=U_{ni}.
\]
Then, using these approximations, we find
\begin{equation} \label{eq:MERB:poly:q}
  \hat{q}_{n}(\tau) = \hat{N}_n (t_n, \hat{u}_n)+ (t_n + \tau) \hat{V}_n  +  \sum_{i=2}^{s} \Big(\sum_{k=1}^{m_i}\dfrac{\beta^{(k)}_{i} }{H^{k-1}(k-1)!}\tau^{k-1} \Big) \widehat{D}_{ni}
\end{equation}
and solve one additional linear ODE
 \begin{equation} \label{eq:MERBb}
  y'_{n+1}(\tau)= \hat{J}_n y_{n+1}(\tau)  +   \hat{q}_{n}(\tau),  \q \tau\in[0, H], \q  y_{n+1}(0)=  \hat{u}_n
 \end{equation}
 to obtain the update
\[
  \hat{u}_{n+1}\approx y_{n+1}(H) \approx v_{n+1}(H)=u_{n+1}.
\]
Since this process can be derived from ExpRB schemes satisfying  \eqref{eq2.8}, we call the resulting methods \eqref{eq:MERBa}--\eqref{eq:MERBb} \emph{Multirate Exponential Rosenbrock (MERB)} methods.  Note that since  $\hat{U}_{n1}$ and $y_{n1}(0)$ do not enter the MERB scheme, for the sake of completeness, one can define $\hat{U}_{n1} =  y_{n1}(0) = \hat{u}_n$.

\begin{remark}\label{remark2.3}
(A comparison with MERK methods).
Based on their structure in \eqref{eq:MERBa}--\eqref{eq:MERBb}, MERB methods have similar structure to MERK methods. Hence, they can retain MERK's interesting features, including very few evaluations of the costly slow components, and they do not require computing matrix functions as ExpRB methods do.  The main difference is that at each integration step MERB methods must update the linearization components $\hat{J}_n$, $\hat{V}_n$, $\hat{N}_n$ and $\hat{D}_{nj}$. However, this increased cost may be balanced by the fact that, due to the property \eqref{eq2.7}, high order MERB methods should require considerably fewer modified ODEs than MERK methods of the same order (see Section~\ref{section:2.4}).
\end{remark}
\subsection{Convergence analysis of MERB methods}\label{sec:analysis}
\subsubsection{Analytical framework}\label{framework}
To analyze the convergence of MERB methods, we employ the abstract framework of strongly continuous semigroups (see, e.g.,  \cite{EN2000,PAZY83}) on a Banach space $X$. Throughout this paper, we denote the norm in X by $\| \cdot \|$.  Let
\begin{equation} \label{eq2.12}
  J = \frac{\partial F}{\partial u}(t, u)
\end{equation}
be the Fr\'echet partial derivative of $F$. We make use of the following assumptions.

{\em Assumption 1. The Jacobian  \eqref{eq2.12}  is the generator of a strongly continuous semigroup  $\ee^{tJ}$ in $X$}. This implies that  there exist constants $C$ and $\omega$ such that
\begin{equation} \label{eq:bound1}
  \left\|\ee^{tJ}\right\| \leq C\ee^{\omega t}, \quad t\geq 0,
\end{equation}
and consequently $\varphi_k(HJ)$, $a_{ij}(HJ)$ and $b_{i}(HJ)$ are bounded operators. \\

{\em Assumption 2. The solution $u:[t_0, T]\to X$  of \eqref{eq1.1} is sufficiently smooth with derivatives in $X$, and $F: [t_0, T]\times X \to X$ is sufficiently Fr\'echet differentiable in a strip along the exact solution to \eqref{eq1.1}.
All derivatives occurring are assumed to be uniformly bounded.}\\

{\em  Stability bound.}
Since $\hat{J}_n =\frac{\partial F} {\partial u}(t_n, \hat{u}_n) $ arising in MERB methods changes at every step, and $\hat{J}_n \approx J_n $, we also employ the following stability bound (for the discrete evolution operators on $X$) of exponential Rosenbrock methods (see \cite[Sect. 3.3]{HOS09}) to have
\begin{equation} \label{eq:stability_bound}
\Bigl\|\prod_{j=0}^{n-k} \ee^{H \hat{J}_{n-j}} \Bigr\| \leq C_\text{\rm S},\qquad t_0\le t_k \le t_n\le T.
\end{equation}
The importance of this bound is that the constant $C_\text{\rm S}$ is uniform in $k$ and $n$, despite the fact that $J_n$ varies from step to step.

\subsubsection{A global error representation of MERB methods}\label{error_expression}
Since MERB methods  \eqref{eq:MERBa}--\eqref{eq:MERBb} result in  a numerical solution $\hat{u}_{n+1}$ which approximates the numerical solution  $u_{n+1}$ of ExpRB methods (as denoted above) at time $t_{n+1}$, we will employ the local errors of ExpRB methods to analyze the global error of MERB methods.  Throughout the paper the following error notations will be used.

$\bullet$ \emph{Global error notation for MERB methods}. We denote the global error at time $t_{n+1}$ of a MERB method as
\begin{equation} \label{eq:error}
\hat{e}_{n+1} = \hat{u}_{n+1} - u(t_{n+1}).
\end{equation}

$\bullet$ \emph{Local error notation for ExpRB methods}.
We denote the local error at $t_{n+1}$ of the base ExpRB method as
\begin{equation} \label{eq:local_error}
 \tilde{e}_{n+1} = \tilde{u}_{n+1} - u(t_{n+1})
 \end{equation}
Here, $\tilde{u}_{n+1}$ is the numerical solution of the base ExpRB method obtained after carrying out one step of \eqref{eq2.6} starting from the exact solution $u(t_n)$ as the initial value, i.e.,
\begin{subequations} \label{eq2.22}
\begin{align}
\tilde{u}_{n+1} = \ee^{H  \tilde{J}_{n}} u(t_n)  + & H \varphi_1 (H  \tilde{J}_{n}) \tilde{V}_{n} t_n+H^2 \varphi_2 (H J_n) \tilde{V}_{n}  \label{eq2.22a} \\
& + H \sum_{i=1}^{s}b_{i}(H  \tilde{J}_{n}) \tilde{N}_n(t_n +c_i H,\tilde{U}_{ni}), \notag \\
\tilde{U}_{ni} = \ee^{c_i H \tilde{J}_{n}}u(t_n)  &+ c_i H \varphi_1 (c_i H  \tilde{J}_{n}) \tilde{V}_{n} t_n+c^2_i H^2 \varphi_2 (c_i H  \tilde{J}_{n}) \tilde{V}_{n}   \label{eq2.22b}  \\
&+ H \sum_{j=1}^{i-1}a_{ij}(H  \tilde{J}_{n}) \tilde{N}_n(t_n +c_j H, \tilde{U}_{nj}),  \notag
\end{align}
\end{subequations}
where
\begin{equation}  \label{eq:JVN_tilde}
\tilde{J}_{n}=\frac{\partial F} {\partial u}(t_n, u(t_n)), \ \tilde{V}_{n}  =\frac{\partial F} {\partial t}(t_n,u(t_n)),  \  \tilde{N}_n (t, u)= F(t,u)- \tilde{J}_{n} u - \tilde{V}_{n} t.
\end{equation}
Note that from Lemma~\ref{Lemma2.1}, \eqref{eq2.22} is equivalent to one step of the MERB scheme starting from the exact initial value $y_{n+1}(0)= u(t_n)$ (for which the solution of the IVP \eqref{eq:MERBb} on $[0, H]$ is ``known'' to be $y_{n+1}(H) = \tilde{u}_{n+1}$). Therefore, one can  consider that $ \tilde{e}_{n+1}$ is also the local error of MERB methods.

$\bullet$ \emph{Global error notation for approximation of the IVP \eqref{eq:MERBb}}.
As $ \hat{u}_{n+1} \approx y_{n+1}(H)$ (the true solution of the ODE \eqref{eq:MERBb}), we denote the global error of an ODE solver used for integrating \eqref{eq:MERBb} on $[0, H]$ as
\begin{equation}  \label{eq2.25}
  \hat{\varepsilon}_{n+1} =\hat{u}_{n+1} -y_{n+1} (H).
\end{equation}

$\bullet$ \emph{Global error notation for approximation of the IVP \eqref{eq:MERBa}}.
Similarly, since $ \hat{U}_{ni}$ is the numerical solution of \eqref{eq:MERBa} on $[0, c_i H]$ obtained by an ODE solver, let us denote the global error of this approximation as
 \begin{equation}\label{eq2.25b}
\hat{\varepsilon}_{ni} =\hat{U}_{ni} - y_{ni} (c_i H).
 \end{equation}

Note that by applying the variation-of-constants formula to \eqref{eq:MERBb}
and using \eqref{eq:coeff_b}, $y_{n+1}(H)$ can be represented as
\begin{equation} \label{eq2.26}
\begin{aligned}
  y_{n+1}(H)= \ee^{H  \hat{J}_{n}} \hat{u}_n  & + H \varphi_1 (H  \hat{J}_{n}) \hat{V}_{n} t_n+H^2 \varphi_2 (H J_n) \hat{V}_{n} \\
  & + H \sum_{i=1}^{s}b_{i}(H \hat{J}_{n}) \hat{N}_n(t_n +c_i H, \hat{U}_{ni}).
\end{aligned}
\end{equation}
In view of  \eqref{eq:error}, \eqref{eq:local_error},  and \eqref{eq2.25}, we can write
\begin{equation}  \label{eq2.27}
  \hat{e}_{n+1} = \hat{u}_{n+1} - \tilde{u}_{n+1} +  \tilde{e}_{n+1} = \hat{\varepsilon}_{n+1} + (y_{n+1} (H)  - \tilde{u}_{n+1}) +  \tilde{e}_{n+1},
\end{equation}
i.e., the global error arising from the MERB scheme can be written as the sum of the global error of the ODE solver used for the IVP \eqref{eq:MERBb}, the difference between
the true solution to the IVP \eqref{eq:MERBb} and the numerical solution obtained by the base ExpRB method \eqref{eq2.22},
and the local error arising from this ExpRB-based method.

To keep our presentation in a compact form, we introduce
\begin{subequations} \label{eq:Bterms}
\begin{align}
t_{ni} & = t_n +c_i H,  \label{eq:tni}\\
\hat{B}_n  &= \varphi_1 (H  \hat{J}_{n}) \hat{V}_{n} t_n+H \varphi_2 (H \hat{J}_n) \hat{V}_{n} +  \sum_{i=1}^{s}b_{i}(H \hat{J}_{n}) \hat{N}_n(t_{ni}, \hat{U}_{ni}),   \label{eq:B_hat}\\
\tilde{B}_n  & = \varphi_1 (H  \tilde{J}_{n}) \tilde{V}_{n} t_n+H \varphi_2 (H \tilde{J}_n) \tilde{V}_{n}   +\sum_{i=1}^{s}b_{i}(H  \tilde{J}_{n}) \tilde{N}_n(t_{ni},\tilde{U}_{ni}). \label{eq:B_tilde}
\end{align}
\end{subequations}
Using \eqref{eq:Bterms}, we now derive  a full expansion of \eqref{eq2.27}, which later tells us how the global error of MERB methods can be estimated by the sum of the propagated local errors of ExpRB methods and the global errors of the ODE solvers used for integrating \eqref{eq:MERBa} and \eqref{eq:MERBb}.
\begin{theorem}\label{Theorem2.4}
The global error  $\hat{e}_{n+1}$  of MERB methods  \eqref{eq:MERBa}--\eqref{eq:MERBb} at time $t_{n+1}$ can be expressed as
 \begin{equation}  \label{eq:error_terms}
 \begin{aligned}
 \hat{e}_{n+1} &
        = \underbrace{\Big(\prod_{j=0}^{n}  \ee^{H  \hat{J}_{n-j}} - \prod_{j=0}^{n}  \ee^{H  \tilde{J}_{n-j}}\Big) u_0}_{Error1}
       + \underbrace{\sum_{k=0}^{n}  \Big(\prod_{j=0}^{n-k-1}  \ee^{H   \tilde{J}_{n-j}} \Big)  \tilde{e}_{k+1} }_{Error2} \\
  &  +  \underbrace{\sum_{k=0}^{n}  \Big(\prod_{j=0}^{\mathsmaller{n-k-1}}  \ee^{H  \hat{J}_{n-j}} \Big)  \hat{\varepsilon}_{k+1}}_{Error3}
      +  \underbrace{H \sum_{k=0}^{n}  \bigg[\Big(\prod_{j=0}^{\mathsmaller{n-k-1}}  \ee^{H   \hat{J}_{n-j}} \Big) \hat{B}_k  - \Big(\prod_{j=0}^{\mathsmaller{n-k-1}}  \ee^{H   \tilde{J}_{n-j}} \Big)  \tilde{B}_k \bigg]}_{Error4}.
 \end{aligned}
 \end{equation}
\end{theorem}
\begin{proof}
In view of \eqref{eq2.27}, we first study the difference $(y_{n+1} (H)  - \tilde{u}_{n+1})$.
Using \eqref{eq:B_hat} and \eqref{eq2.25} (which implies $\hat{u}_n = y_n (H) +  \hat{\varepsilon}_n $), we have
\begin{equation} \label{eq2.32}
  y_{n+1}(H)= \ee^{H  \hat{J}_{n}} \hat{u}_n  + H \hat{B}_n =  \ee^{H  \hat{J}_{n}} y_n (H) +  \ee^{H  \hat{J}_{n}} \hat{\varepsilon}_n  + H \hat{B}_n.
\end{equation}
Solving this recurrence relation (with note that $y_0 (H) = u(t_0) =u_0$) gives
\begin{equation} \label{eq2.33}
  y_{n+1}(H)=  \Big(\prod_{j=0}^{n}  \ee^{H  \hat{J}_{n-j}} \Big) u_0  + \sum_{k=0}^{n-1}  \Big(\prod_{j=0}^{n-k-1}  \ee^{H  \hat{J}_{n-j}} \Big)  \hat{\varepsilon}_{k+1} +H \sum_{k=0}^{n}  \Big(\prod_{j=0}^{n-k-1} \ee^{H   \hat{J}_{n-j}} \Big) \hat{B}_k .
\end{equation}
Similarly, using \eqref{eq:B_tilde} and \eqref{eq:local_error} (which implies $u(t_n) = \tilde{u}_n -  \tilde{e}_n $), we can write $\tilde{u}_{n+1}$ in \eqref{eq2.22a} as
\begin{equation} \label{eq2.34}
\tilde{u}_{n+1} =  \ee^{H   \tilde{J}_{n}}  u(t_n)  + H  \tilde{B}_n =  \ee^{H   \tilde{J}_{n}} \tilde{u}_n -   \ee^{H   \tilde{J}_{n}} \tilde{e}_n    + H  \tilde{B}_n.
\end{equation}
After solving this recurrence, we end up with
\begin{equation} \label{eq2.35}
\tilde{u}_{n+1}  =  \Big(\prod_{j=0}^{n}  \ee^{H  \tilde{J}_{n-j}} \Big) u_0  -  \sum_{k=0}^{n-1}  \Big(\prod_{j=0}^{n-k-1}  \ee^{H  \tilde{J}_{n-j}} \Big)   \tilde{e}_{k+1} + H \sum_{k=0}^{n}  \Big(\prod_{j=0}^{n-k-1}  \ee^{H    \tilde{J}_{n-j}} \Big)  \tilde{B}_k .
\end{equation}
Subtracting  \eqref{eq2.35} from \eqref{eq2.33} gives  $(y_{n+1} (H)  - \tilde{u}_{n+1})$ and inserting this into \eqref{eq2.27} proves \eqref{eq:error_terms}.
\end{proof}
Next, in order to estimate  the global error $ \hat{e}_{n+1}$, we prove some preliminary results.
\subsubsection{Preliminary results and error bounds}\label{pre_results}
\begin{lemma}\label{Lemma2.2}
The term Error4  in \eqref{eq:error_terms} can be further expressed as
\begin{equation} \label{eq:Error4}
Error4 = H \sum_{k=0}^{n}  \bigg[\Big(\prod_{j=0}^{n-k-1}  \ee^{H   \hat{J}_{n-j}} - \prod_{j=0}^{n-k-1}  \ee^{H   \tilde{J}_{n-j}} \Big) \hat{B}_k + \Big(\prod_{j=0}^{n-k-1}  \ee^{H   \tilde{J}_{n-j}} \Big)  (\hat{B}_k - \tilde{B}_k ) \bigg],
\end{equation}
where
\begin{equation} \label{eq:diffB}
\begin{aligned}
\hat{B}_k - \tilde{B}_k &= \sum_{j=1}^{2}  \big[\big(\varphi_j (H  \hat{J}_{k}) -  \varphi_j (H  \tilde{J}_{k})\big) \tilde{V}_{k} + \varphi_j (H  \hat{J}_{k}) ( \hat{V}_k - \tilde{V}_k ) \big] t^{2-j}_k H^{j-1} \\
  & +   \sum_{i=1}^{s} \big( b_{i}(H \hat{J}_{k}) - b_{i}(H  \tilde{J}_{k}) \big) \tilde{N}_k (t_{ki},\tilde{U}_{ki}) \\
  & +   \sum_{i=1}^{s}  b_{i}(H  \hat{J}_{k}) \big( \hat{N}_k (t_{ki}, \hat{U}_{ki}) - \tilde{N}_k (t_{ki},\tilde{U}_{ki})  \big).
\end{aligned}
\end{equation}
\end{lemma}
\begin{proof}
The derivation of \eqref{eq:Error4} is straightforward by subtracting and adding the same term $\prod_{j=0}^{n-k-1}  \ee^{H   \tilde{J}_{n-j}}  \hat{B}_k $ within the sum $ \sum_{k=0}^{n}  \big[ \cdot \big]$ in \emph{Error4}. Also, by subtracting \eqref{eq:B_tilde} from  \eqref{eq:B_hat}, one can easily obtain \eqref{eq:diffB}.
\end{proof}
To estimate the difference in the nonlinear terms at each internal MERB and ExpRB stage, $\big( \hat{N}_k (t_{ki}, \hat{U}_{ki}) - \tilde{N}_k (t_{ki},\tilde{U}_{ki})  \big)$ in \eqref{eq:diffB}, we first study the difference
\begin{equation} \label{eq:E_hat}
\hat{E}_{ni} = \hat{U}_{ni} -  \tilde{U}_{ni}.
\end{equation}
Denoting
\begin{subequations} \label{eq:Aterms}
\begin{align}
\hat{A}_{ni}  &= c_i \varphi_1 (c_i H  \hat{J}_{n}) \hat{V}_{n} t_n+ c^2_i H \varphi_2 (c_i H J_n) \hat{V}_{n} +  \sum_{j=1}^{i-1}a_{ij}(H \hat{J}_{n}) \hat{N}_n(t_{nj}, \hat{U}_{nj}),   \label{eq:A_hat}\\
 \tilde{A}_{ni}  &= c_i \varphi_1 (c_i H  \tilde{J}_{n})  \tilde{V}_{n} t_n+ c^2_i H \varphi_2 (c_i H J_n)  \tilde{V}_{n} +  \sum_{j=1}^{i-1}a_{ij}(H  \tilde{J}_{n})  \tilde{N}_n(t_{nj},  \tilde{U}_{nj}), \label{eq:A_tilde}
\end{align}
\end{subequations}
we obtain the following result.
\begin{lemma}\label{Lemma2.3}
 The difference  between  $\hat{U}_{ni}$  and  $\tilde{U}_{ni}$ can be expressed as
\begin{equation} \label{eq:E_hat_result}
 \hat{E}_{ni} = \hat{\varepsilon}_{ni} +    \ee^{c_i H   \hat{J}_n} \hat{e}_n + \big( \ee^{c_i H   \hat{J}_n} -  \ee^{c_i H   \tilde{J}_n} \big)u(t_n) +  H (\hat{A}_{ni} -   \tilde{A}_{ni})
 \end{equation}
 with
 \begin{equation} \label{eq:diffA}
\begin{aligned}
\hat{A}_{ni} - \tilde{A}_{ni} &= \sum_{\ell=1}^{2}  \big[\big(\varphi_{\ell} (c_i H  \hat{J}_{n}) -  \varphi_{\ell} (c_i H  \tilde{J}_{n})\big) \tilde{V}_{n} + \varphi_{\ell} (c_i H  \hat{J}_{n}) ( \hat{V}_n - \tilde{V}_n ) \big] c^{\ell}_i t^{2-{\ell} }_n H^{{\ell} -1}  \\
  & +   \sum_{j=1}^{i-1} \big( a_{ij
  }(H \hat{J}_{n}) - a_{ij}(H  \tilde{J}_{n}) \big) \tilde{N}_n (t_{nj},\tilde{U}_{nj}) \\
  & +  \sum_{j=1}^{i-1}  a_{ij}(H \hat{J}_{n}) \big( \hat{N}_n (t_{nj}, \hat{U}_{nj}) - \tilde{N}_n (t_{nj},\tilde{U}_{nj})  \big).
\end{aligned}
\end{equation}
Here, $\hat{\varepsilon}_{n1} =\hat{U}_{n1} - y_{n1} (c_1 H) = \hat{u}_n - y_{n1} (0) =0$  (due to $c_1 =0$) and thus $\hat{E}_{n1} = \hat{e}_n $.
\end{lemma}
\begin{proof}
From \eqref{eq:E_hat} and \eqref{eq2.25b}, we have
 \begin{equation} \label{eq2.42}
 \hat{E}_{ni} =  \hat{\varepsilon}_{ni} + y_{ni}(c_i H) -  \tilde{U}_{ni}.
 \end{equation}
Using \eqref{eq:A_tilde}, one can write $\tilde{U}_{ni}$ given in \eqref{eq2.22b} as
\begin{equation} \label{eq2.43}
\tilde{U}_{ni} = \ee^{c_i H  \tilde{J}_{n}} u(t_n)  +  H \tilde{A}_{ni}.
\end{equation}
By  applying the variation-of-constants formula to \eqref{eq:MERBa}
and using \eqref{eq:coeff_a},
\begin{equation} \label{eq2.44}
y_{ni}(c_i H)= \ee^{c_i H  \hat{J}_{n}} \hat{u}_n  +  H \hat{A}_{ni} =  \ee^{c_i H  \hat{J}_{n}} (\hat{e}_n + u(t_n)) +  H \hat{A}_{ni},
\end{equation}
where $\hat{A}_{ni}$ is given in  \eqref{eq:A_hat}.
Inserting \eqref{eq2.43} and \eqref{eq2.44} into \eqref{eq2.42} gives \eqref{eq:E_hat_result}.
Similarly to \eqref{eq:diffB}, the expression \eqref{eq:diffA} can be verified by subtracting \eqref{eq:A_tilde} from
\eqref{eq:A_hat} first and then adding and subtracting to the result the same terms $c_i \varphi_1 (c_i H  \hat{J}_{n})  \tilde{V}_{n} t_n$, $c^2_i H \varphi_2 (c_i H  \hat{J}_{n})  \tilde{V}_{n} $, and  $\sum_{j=1}^{i-1}  a_{ij}(H \hat{J}_{n})  \tilde{N}_n (t_n +c_j H,\tilde{U}_{nj})$.
\end{proof}
Next, we prove several bounds needed to estimate the terms in \eqref{eq:diffB} and \eqref{eq:diffA}. To simplify our presentation within both this and the following subsections, we will use $C$ as a generic constant that may have different values at each occurrence.
\begin{lemma}\label{Lemma2.4}
Under Assumption~2, the bound
\begin{equation} \label{eq:bound_diff_N}
\|\hat{N}_n (t_{ni}, \hat{U}_{ni}) - \tilde{N}_n (t_{ni},\tilde{U}_{ni}) \| \leqslant C \|\hat{E}_{ni}\| + C \|\hat{E}_{ni}\| ^2 +  C\| \hat{J}_n - \tilde{J}_n\|+ C\| \hat{V}_n - \tilde{V}_n\|
\end{equation}
holds for all $n$ and $i$  as long as $\hat{E}_{ni}$ remains in a sufficiently small neighborhood of $0$.
\end{lemma}
\begin{proof}
First, we split
\small
\begin{equation} \label{eq2.46} \notag
\hat{N}_n (t_{ni}, \hat{U}_{ni}) - \tilde{N}_n (t_{ni},\tilde{U}_{ni})
= \underbrace{\hat{N}_n (t_{ni}, \hat{U}_{ni}) - \hat{N}_n (t_{ni},\tilde{U}_{ni})}_{Nsplit1}
+  \underbrace{\hat{N}_n (t_{ni},\tilde{U}_{ni}) -  \tilde{N}_n (t_{ni},\tilde{U}_{ni})}_{Nsplit2}.
\end{equation}
\normalsize
Using \eqref{eq:JVN_hat} and \eqref{eq:JVN_tilde}, we write the term \emph{Nsplit2} as
\begin{equation} \label{eq2.47}
\begin{aligned}
Nsplit2 &= \big(F(t_{ni}, \tilde{U}_{ni}) - \hat{J}_n  \tilde{U}_{ni} - \hat{V}_n t_{ni}\big) - \big( F(t_{ni}, \tilde{U}_{ni}) -  \tilde{J}_n  \tilde{U}_{ni} -  \tilde{V}_n t_{ni} \big)\\
& =  (\tilde{J}_n -  \hat{J}_n)  \tilde{U}_{ni} + ( \tilde{V}_n - \hat{V}_n ) t_{ni}.
\end{aligned}
\end{equation}
 Expanding $\hat{N}_n(t_{ni},\hat{U}_{ni})$ into a Taylor series expansion around $(t_{ni},\tilde{U}_{ni})$ gives
\begin{equation} \label{eq2.48}
Nsplit1 = \frac{\partial \hat{N}_n}{\partial u}(t_{ni},  \tilde{U}_{ni}) \hat{E}_{ni}  + \int_{0}^{1} (1-\theta ) \frac{\partial^2  \hat{N}_n}{\partial u^2}(t_{ni}, \tilde{U}_{ni}+ \theta \hat{E}_{ni})(\hat{E}_{ni}, \hat{E}_{ni})\dd\theta.
\end{equation}
Under Assumption~2, \eqref{eq:bound_diff_N}  follows by bounding $\|Nsplit1\| + \|Nsplit2\|$.
\end{proof}
\begin{lemma}\label{Lemma2.5}
Under Assumptions~1 and 2, the bounds
\begin{subequations}  \label{eq:bounds}
\begin{align}
\| \hat{J}_n - \tilde{J}_n\| &  \leqslant C\|\hat{e}_n\|+ C \|\hat{e}_n\|^2,    \label{eq:bound:diffJ} \\
\| \hat{V}_n - \tilde{V}_n\| &  \leqslant C\|\hat{e}_n\|+ C \|\hat{e}_n\|^2,   \label{eq:bound:diffV} \\
\|\ee^{t \hat{J}_n}-\ee^{t\tilde{J}_n}\| &  \leqslant  C t \|\hat{e}_n\|, \q t\geq 0        \label{eq:bound:diffExp} \\
 \|\varphi _{\ell} ( t \hat{J}_n) - \varphi _{\ell} ( t \tilde{J}_n)\|  & \leqslant C t \|\hat{e}_n\|, \q t\geq 0  \label{eq:bound:diffVarphi} \\
 \|b_{i} ( H \hat{J}_n) - b _{i} (H  \tilde{J}_n)\|  & \leqslant C H \|\hat{e}_n\|, \label{eq:bound:diff_bi} \\
 \|a_{ij} ( H \hat{J}_n) - a_{ij}  (H  \tilde{J}_n)\|  & \leqslant C H \|\hat{e}_n\|   \label{eq:bound:diff_aij}
\end{align}
\end{subequations}
hold for all $n$, $\ell$, $i$ and $j$, as long as the global errors $\hat{e}_n$ remain in a sufficiently small neighborhood of $0$.
\end{lemma}
\begin{proof}
By definition, $\hat{J}_n - \tilde{J}_n =  \frac{\partial F}{\partial u}(t_n, \hat{u}_n) - \frac{\partial F}{\partial u}(t_n, u(t_n))$.  Using Assumption~2, one can expand  $G(t, u):= \frac{\partial F}{\partial u}(t, u)$ in a Taylor series around $(t_n, u(t_n))$ to get
\[
  \hat{J}_n - \tilde{J}_n =  \frac{\partial G}{\partial u}(t_n,  u(t_n)) \hat{e}_n + \mathcal{O}(\|\hat{e}_n\|^2),
\]
which shows \eqref{eq:bound:diffJ}. Similarly \eqref{eq:bound:diffV} may be verified by expanding $\frac{\partial F}{\partial t}(t, u)$ in a Taylor series around $(t_n, u(t_n))$.

Next, we estimate the difference between the two semigroups $\ee^{t \hat{J}_n}$ and $\ee^{t\tilde{J}_n}$ in a similar manner as in \cite[Lemma~4.2]{LO14a}. Observing that
$\ee^{t \hat{J}_n}$ is the solution of the IVP
\begin{equation*}
w'(t)=\hat{J}_n w(t) = \tilde{J}_n w(t) + (\hat{J}_n - \tilde{J}_n) w(t), \q w(0) = I,
\end{equation*}
We apply the variation-of-constants formula to this IVP to obtain
\begin{equation*}
\ee^{t \hat{J}_n}-\ee ^{t\tilde{J}_n}= t\int_{0}^{1} \ee^{(1-\theta) t  \tilde{J}_n} (\hat{J}_n - \tilde{J}_n)\ee^{\theta t \hat{J}_n} \dd\theta.
\end{equation*}
Therefore, \eqref{eq:bound:diffExp} follows directly from \eqref{eq:bound1} and \eqref{eq:bound:diffJ}.  Using this, the bounds \eqref{eq:bound:diffVarphi}--\eqref{eq:bound:diff_aij} follow from using \eqref{eq2.4} and \eqref{eq2.8} (see also \cite[Lemma~4.3]{LO14a}).
\end{proof}
Using the results from Lemmas~\ref{Lemma2.3}, \ref{Lemma2.4}, and \ref{Lemma2.5}, we obtain the following result.
\begin{corollary}\label{Corollary2.1}
Under Assumptions~1 and 2, the estimate
\begin{equation}  \label{eq:bound:diffB}
\| \hat{B}_k - \tilde{B}_k \|   \leqslant     \sum_{j=1}^{i} C \|\hat{\varepsilon}_{kj} \| + C  \|\hat{e}_k\|  + C  \|\hat{e}_k\|^2
\end{equation}
holds for all $k$, as long as $\hat{E}_{ki} $ and the global errors $\hat{e}_k$ remain in a sufficiently small neighborhood of $0$.
\end{corollary}
\begin{proof}
Using Lemmas~\ref{Lemma2.5} and \ref{Lemma2.4}, one can bound \eqref{eq:diffB} as
\begin{equation} \label{eq2.52}
\| \hat{B}_k - \tilde{B}_k \|    \leqslant   C H \|\hat{e}_k\|  + C  \|\hat{e}_k\| + C  \|\hat{e}_k\|^2 + C \|\hat{E}_{ki}\| + C \|\hat{E}_{ki}\| ^2.
\end{equation}
Next, we apply Lemma~\ref{Lemma2.3} (with $n = k$) to get $\hat{E}_{ki}$ and then estimate it by using  \eqref{eq:bound1}  and Lemma~\ref{Lemma2.5} (the bound \eqref{eq:bound:diffExp}):
\begin{equation} \label{eq2.53}
\|\hat{E}_{ki} \|   \leqslant   \|\hat{\varepsilon}_{ki} \| + C  \|\hat{e}_k\|  + C H \|\hat{e}_k\|  +  H \| \hat{A}_{ki} -  \tilde{A}_{ki} \|.
\end{equation}
Again using Lemmas~\ref{Lemma2.5} and \ref{Lemma2.4}, the bound on $\| \hat{A}_{ki} -  \tilde{A}_{ki} \|$ from \eqref{eq:diffA} is similar to \eqref{eq2.52}.  Inserting this into \eqref{eq2.53} we have
\begin{equation} \label{eq2.54}
\|\hat{E}_{ki} \|   \leqslant   \|\hat{\varepsilon}_{ki} \| +  C H \|\hat{e}_k\|  + C  \|\hat{e}_k\|  +  C  \|\hat{e}_k\|^2  +  \sum_{j=1}^{i-1} C \|\hat{E}_{kj}\|.
\end{equation}
Since $\hat{E}_{k1} = \hat{e}_k$ (see Lemma~\ref{Lemma2.3}), this relation finally shows that
\begin{equation} \label{eq2.55}
\|\hat{E}_{ki} \|   \leqslant   \|\hat{\varepsilon}_{ki} \| + C H \|\hat{e}_k\|  + C  \|\hat{e}_k\|  + C  \|\hat{e}_k\|^2  +  \sum_{j=1}^{i-1} C \|\hat{\varepsilon}_{kj} \|.
\end{equation}
Now using the fact that $C H \|\hat{e}_k\|  + C  \|\hat{e}_k\|   = (C H + C) \|\hat{e}_k\|   \leqslant  C \|\hat{e}_k\| $, one can easily show \eqref{eq:bound:diffB} from
 \eqref{eq2.52} and  \eqref{eq2.55}.
 \end{proof}
Finally, we  give a technical lemma, which can be later used to estimate the term \emph{Error1} appearing in  \eqref{eq:error_terms}.
\begin{lemma}\label{Lemma2.6}
Let $\{Z_j\}_{j=0}^{n}$ and  $\{Y_j \}_{j=0}^{n}$  be two sequences (of operators) in $X$. We have
\begin{equation} \label{eq2.56}
\prod_{j=0}^{n} Z_{n-j} - \prod_{j=0}^{n} Y_{n-j} =  \sum_{k=0}^{n}  \Big(\prod_{j=0}^{n-k-1} Z_{n-j}  \Big) (Z_k - Y_k ) \Big(\prod_{j= n-k +1}^{n} Y_{n-j}  \Big).
\end{equation}
\end{lemma}
\begin{proof}
By adding and subtracting $\prod_{j=0}^{n-1} Z_{n-j}Y_0$ and then $\prod_{j=0}^{n-2} Z_{n-j}Y_0Y_1$, the left hand side of \eqref{eq2.56} can be written as
\small
\begin{equation} \label{eq2.57}\notag
\begin{aligned}
& Z_n  Z_{n-1} \dots Z_1 (Z_0 - Y_0)  + Z_n  Z_{n-1} \dots Z_2 (Z_1 - Y_1) Y_0
 + (Z_n  Z_{n-1} \dots Z_2 - Y_n  Y_{n-1} \dots Y_2) Y_1 Y_0 \\
& =  \Big(\prod_{j=0}^{n-1} Z_{n-j} \Big) (Z_0 - Y_0)  + \Big(\prod_{j=0}^{n-2} Z_{n-j} \Big)(Z_1 - Y_1) Y_0 + (Z_n  Z_{n-1} \dots Z_2 - Y_n  Y_{n-1} \dots Y_2) Y_1 Y_0.
\end{aligned}
\end{equation}
 \normalsize
We continue adding and subtracting $\left(\prod_{j=0}^{n-k-1} Z_{n-j}\right)\left(\prod_{j=n-k+1}^n Y_{n-j}\right)$ in this manner until $k=n$ to obtain the right hand side \eqref{eq2.56}. 
\end{proof}
\begin{corollary}\label{Corollary2.2}
Under Assumptions~1 and 2, the estimate
\begin{equation} \label{eq2.58}
\Big\|\prod_{j=0}^{n}  \ee^{H  \hat{J}_{n-j}} - \prod_{j=0}^{n}  \ee^{H  \tilde{J}_{n-j}}\Big\|  \leqslant  H \sum_{k=0}^{n} C \|\hat{e}_k\|.
\end{equation}
holds for all $n$ as long as the global errors $\hat{e}_k$ remain sufficiently small.
\end{corollary}
\begin{proof}
This follows by applying Lemma~\ref{Lemma2.6} to $Z_{n-j} = \ee^{H \hat{J}_{n-j}}$ and $Y_{n-j} = \ee^{H \tilde{J}_{n-j}}$, and using the stability bound \eqref{eq:bound1} and the bound \eqref{eq:bound:diffExp} from Lemma~\ref{Lemma2.5}.
\end{proof}

\subsubsection{MERB convergence}\label{sec:convergence}

With the above preparation in hand, we are now ready to prove convergence of our MERB methods.

\begin{theorem}\label{convergence}
Let the initial value problem \eqref{eq1.1} satisfy Assumptions~1--2.  Consider for its numerical solution a MERB method  \eqref{eq:MERBa}--\eqref{eq:MERBb} that is constructed from an ExpRB method of global order $p$ using with macro time step $H$.  Let  $m$ denote the number of fast steps per slow step. If the fast ODEs \eqref{eq:MERBa} and \eqref{eq:MERBb} associated with the MERB method are integrated with micro time step
$h=H/m$ by using ODE solvers that have global order of convergence $q$ and $r$, respectively, then the MERB method is convergent with the error bound
\begin{equation}\label{eq:global_error}
  \| \hat{u}_n -u(t_n) \| \leqslant  C H^p + C Hh^q + C h^r
\end{equation}
on compact time intervals \ $t_0 \leq  t_n =t_0+nH \leq  T$. Here, while the first error constant  depends on $T-t_0$ (but is independent of $n$ and
$H$), the second and third error constants also depend on the error constants of the chosen ODE solvers.
\end{theorem}
\begin{proof}
We begin with the global error expansion given in Theorem~\ref{Theorem2.4}, and estimate each of the terms in \eqref{eq:error_terms}.
First, from Corollary~\ref{Corollary2.2} it is obvious that $\|Error1\| \leqslant  H \sum_{k=0}^{n} C  \|\hat{e}_k\|$.
Then the stability bound \eqref{eq:bound1} tells us that $\|Error2\| \leqslant  \sum_{k=0}^{n} C \| \tilde{e}_{k+1}\|$ and  $\|Error3\| \leqslant  \sum_{k=0}^{n} C \| \hat{\varepsilon}_{k+1}\|$. Next, examining the expression \eqref{eq:Error4} we employ Corollaries \ref{Corollary2.1} and \ref{Corollary2.2}, along with the stability bound \eqref{eq:bound1}, to obtain
$\|Error4\| \leqslant  H \sum_{k=0}^{n} \big[ CH \|\hat{e}_k\| +  \sum_{j=1}^{i} C \|\hat{\varepsilon}_{kj} \| + C  \|\hat{e}_k\|  + C  \|\hat{e}_k\|^2 \big]$.
Therefore,  we derive from \eqref{eq:error_terms} that
\begin{equation}\label{eq:2.60}
 \|\hat{e}_{n+1} \|  \leqslant  H \sum_{k=0}^{n} C  \|\hat{e}_k\| +  \sum_{k=0}^{n} C \| \tilde{e}_{k+1}\| + \sum_{k=0}^{n} C \| \hat{\varepsilon}_{k+1}\| + H \sum_{k=0}^{n} (\sum_{j=1}^{i} C \|\hat{\varepsilon}_{kj} \|).
\end{equation}
From our assumption that the base ExpRB method has global order $p$, its local error satisfies $\|\tilde{e}_{k+1}\|  \leqslant  C H^{p+1}$.

As for the global errors $\hat{\varepsilon}_{k+1}$ and $\hat{\varepsilon}_{kj}$ obtained by solving the fast ODEs  \eqref{eq:MERBb} and  \eqref{eq:MERBa}  on $[0, H]$ and  $[0, c_i H]$, respectively (using micro time step $h$), the global error analysis from \cite[Theorem~3.4]{hairer93} guarantees that
\begin{subequations} \label{fastError}
 \begin{align}
   \|\hat{\varepsilon}_{k+1} \|  & \leqslant  h^r \tfrac{C}{L} (e^{ L H} -1) = C h^r H \tfrac{e^{ L H} -1}{LH}  =  Ch^r H \varphi_1 (L H) \leqslant Ch^r H,  \label{fastError:a} \\
   \|\hat{\varepsilon}_{kj} \|  & \leqslant  h^q \tfrac{C}{L} (e^{ L c_i H} -1) = C h^q c_i H \tfrac{e^{ L c_i H} -1}{L c_i H}  \leqslant    Ch^q H \varphi_1 (L c_i H)  \leqslant Ch^q H.  \label{fastError:b}
 \end{align}
\end{subequations}
These bounds require that the Jacobian $\hat{J}_k$ of the right hand sides of both ODEs satisfies $\|\hat{J}_k \| \leqslant  L$.  This follows from $\|\frac{\partial F}{\partial u}(t, u) \| \leqslant L$, which easily follows from Assumption~2.  Combining these bounds and shifting the index $n$ in \eqref{eq:2.60} to $n-1$, we obtain
\begin{equation}\label{eq:2.62}
 \|\hat{e}_{n} \|  \leqslant  H \sum_{k=0}^{n-1} C  \|\hat{e}_k\| +  \sum_{k=0}^{n-1} \big( C H^{p+1} +  C h^r H  + C h^q H^2 \big).
\end{equation}
The error bound \eqref{eq:global_error} results from applying a discrete Gronwall lemma to \eqref{eq:2.62}.
\end{proof}
\begin{remark} \label{remark2.6}
Since $h=H/m$, one can write the error bound \eqref{eq:global_error} as $\| \hat{u}_n -u(t_n) \| \leqslant  C H^p + \frac{ C}{m^q} H^{q+1} +  \frac{C}{m^r}H^r$. Thus for a fixed $m$, a MERB method \eqref{eq:MERBa}--\eqref{eq:MERBb} will converge with order $p$ provided that the inner ODE solvers for \eqref{eq:MERBa} and \eqref{eq:MERBb} have orders $q \ge p-1$ and $r\ge p$, respectively. We note that this is an improvement compared to MRI-GARK methods \cite{Sandu2019}, that in fact require both $q \ge p$ and $r\ge p$ for a method of order $p$. It is also worth mentioning that the error bound   \eqref{eq:global_error} for MERB methods is similar to the one obtained with MERK methods \cite{LCR2020}.
\end{remark}
\subsection{Construction of specific MERB methods}
\label{section:2.4}

Guided by Theorem~\ref{convergence}, in order to derive MERB methods it is important to begin with base ExpRB methods that satisfy Lemma~\ref{Lemma2.1}. Fortunately, such ExpRB methods are available up to order 6 in the literature, see \cite{HOS09,LO14a,LO16}. In this subsection, we extend some of these methods to write their coefficients more generally, and then derive MERB methods of orders 2 through 6 from these schemes.  Note that since a MERB method \eqref{eq:MERBa}--\eqref{eq:MERBb} is uniquely characterized by its polynomials $\hat{p}_{ni}(\tau)$ and $\hat{q}_n(\tau)$, we only provide those polynomials here.  In particular, we note that these MERB methods require fewer modified ODEs to be solved per slow time step than comparable order MRI-GARK  \cite{Sandu2019} and MERK methods \cite{LCR2020}.
\subsubsection{Second-order methods}
First, consider the second-order ExpRB-Euler scheme (see \cite{HOS09}, and \cite[Sect. 1.2.2]{Luan2014} for non-autonomous problems)
\[
  u_{n+1}= u_n  + H \varphi _{1} ( H J_n)F(t_n, u_n) + H^2 \varphi _{2} ( H J_n) V_n.
\]
Using Lemma~\ref{Lemma2.1} we immediately derive from this the second-order $\mathtt{MERB2}$ method:
\begin{align}
\hat{q}_n(\tau) &= \nnltermn + (t_n + \tau) \vn, \quad \tau \in [0,H].
\end{align}
This only requires the solution of one modified ODE.  We note that since second order multirate methods have been available for some time, we do not include $\mathtt{MERB2}$ in our numerical results, and instead focus on higher order multirate methods.

\subsubsection{Third-order methods}
In \cite{HOS09}, a 2-stage 3rd-order ExpRB method called $\mathtt{exprb32}$ was constructed (using $c_2=1$) for autonomous problems.  Extending this to non-autonomous problems and writing this for general $c_2$, we solve condition~1 of Table~\ref{tb2.1} directly (with $s=2$) to give a general family of third-order methods:
\begin{equation} \label{eq:expRB3}
\begin{aligned}
U_{n2} &= u_n  + c_2 H \varphi _{1} ( c_2 H J_n)F(t_n, u_n) +  c^{2}_{2} H^2 \varphi _{2} ( c_2 H J_n)V_n, \\
  u_{n+1}& = u_n  + H \varphi _{1} ( H J_n)F(t_n, u_n) + H^2 \varphi _{2} ( H J_n) V_n + H \tfrac{2}{c^2_2} \varphi_{3} (H J_n) D_{n2}.
\end{aligned}
\end{equation}
From this we construct the $\mathtt{MERB3}$ family of third-order methods:
\begin{equation} \label{eq:MERB3}
\begin{aligned}
   \hat{p}_{n2}(\tau) &= \nnltermn + (t_n + \tau)\vn,  & \tau &\in [0,c_2H],\\
   \hat{q}_n(\tau) &= \nnltermn + (t_n + \tau) \vn +  \frac{\tau^2}{c_2^2H^2}\Dntwo, & \tau &\in [0,H].
\end{aligned}
\end{equation}
Clearly, this requires the solution of 2 modified ODEs per slow time step (whereas third-order MERK and MRI-GARK methods require solving 3 modified ODEs per step).
In our numerical experiments we take $c_2 = \tfrac{1}{2}$, which gives rise to a total fast time step traversal for $\mathtt{MERB3}$ of $(1+c_2)H = 1.5H$.
\subsubsection{Fourth-order method}
There exist several 4th-order ExpRB schemes \cite{HOS09,LO14a,LO16,Luan17,Luan2021b} with coefficients fulfilling  Lemma~\ref{Lemma2.1}. However, we chose a 2-stage 4th-order ExpRB method called $\mathtt{exprb42}$ which was constructed for autonomous problems in \cite{Luan17}. Transforming this to non-autonomous form, we have
\begin{equation} \label{eq:expRB4}
\begin{aligned}
U_{n2} &= u_n  + \tfrac{3}{4} H \varphi _{1} ( \tfrac{3}{4}  H J_n)F(t_n, u_n) +  \tfrac{9}{16}  H^2 \varphi _{2} ( \tfrac{3}{4}  H J_n)V_n, \\
  u_{n+1}& = u_n  + H \varphi _{1} ( H J_n)F(t_n, u_n) + H^2 \varphi _{2} ( H J_n) V_n + H  \tfrac{16}{9} \varphi_{3} (H J_n) D_{n2}.
\end{aligned}
\end{equation}
We then apply Lemma~\ref{Lemma2.1} to construct the 4th-order $\mathtt{MERB4}$ method:
\begin{equation} \label{eq:MERB4}
\begin{aligned}
     \hat{p}_{n2}(\tau) &= \nnltermn + (t_n + \tau)\vn, & \tau &\in [0,\frac{3}{4}H]\\
     \hat{q}_n(\tau) &= \nnltermn + (t_n + \tau)\vn +  \frac{16}{9}\frac{\tau^2}{H^2}\Dntwo, & \tau &\in [0,H].
\end{aligned}
\end{equation}
$\mathtt{MERB4}$ only requires solving 2 modified ODEs per slow time step, whereas 4th-order MRI-GARK and MERK methods require 5 and 4 modified ODEs in each step, respectively.  We further note that \eqref{eq:MERB4} has a total fast traversal time of $\frac{7}{4}H = 1.75H$.

\subsubsection{Fifth-order methods}
ExpRB methods of order 5 can be found in \cite{LO14a,LO16}. Here, for efficiency purposes, we consider a parallel scheme called $\mathtt{pexprb54s4}$, whose coefficients (with fixed nodes $c_i$) satisfy Lemma~\ref{Lemma2.1}. It uses $s =4$ stages and is embedded with a fourth-order scheme (for stepsize adaptivity) but can be implemented as a 3-stage method.  A detailed derivation of $\mathtt{pexprb54s4}$ is given in \cite{LO16} (solving conditions 1--4 of Table~\ref{tb2.1} with the choices $b_2 (Z) =0$, $a_{43}(Z) = 0$, $a_{32}(Z)=\tfrac{2c^3_3}{c^2_2} \varphi_{3}(c_3 Z)$, and $a_{42}=\tfrac{2c^3_4}{c^2_2} \varphi_{3}(c_4 Z)$).
Following that derivation, we present here a family of fifth-order ExpRB methods (depending on parameters $c_2, c_3, c_4$) for non-autonomous problems:
\begin{equation} \label{eq:expRB5}
\begin{aligned}
U_{n2} &= u_n  + H\left(c_2\varphi _{1} (c_2  H J_n)F(t_n, u_n) +  c_2^2 H \varphi _{2} (c_2  H J_n)V_n\right), \\
U_{n3} &= u_n  + H\left( c_3 \varphi _{1} (c_3  H J_n)F(t_n, u_n) +  c_3^2 H \varphi _{2} (c_3  H J_n)V_n + \tfrac{2c^3_3}{c^2_2} \varphi_{3}(c_3 H J_n) D_{n2}\right), \\
U_{n4} &= u_n  + H\left(c_4 \varphi _{1} (c_4  H J_n)F(t_n, u_n) +  c_4^2 H \varphi _{2} (c_4  H J_n)V_n + \tfrac{2c^3_4}{c^2_2} \varphi_{3}(c_4 H J_n) D_{n2}\right), \\
  u_{n+1}& = u_n  + H \left(\varphi _{1} ( H J_n)F(t_n, u_n) + H \varphi _{2} ( H J_n) V_n
               +  b_3 (H J_n)  D_{n3} + b_4 (H J_n) D_{n4}\right)\\
  \text{with} \\
  &  b_3 (H J_n)  = \tfrac{1}{c_3^2 (c_4 - c_3)} \big(c_4 \varphi_{3} (H J_n) - 6 \varphi_{4} (H J_n) \big), \\
   & b_4 (H J_n) =  \tfrac{1}{c_4^2(c_3-c_4)} \big(2 c_3 \varphi_{3} (H J_n) -  6 \varphi_{4} (H J_n) \big),\\
   & c_4=\tfrac{3(5c_3-4)}{5(4c_3-3)}.
\end{aligned}
\end{equation}
We note that the two internal stages $\{ U_{n3}, U_{n4} \}$  are independent of one another and thus can be computed simultaneously. They also have the same format, in that they have the same formula but only act on different inputs $c_3$ and $c_4$, which we exploit below to give the same polynomial for their corresponding modified ODEs.

Applying Lemma~\ref{Lemma2.1} to \eqref{eq:expRB5} results in the fifth-order family of $\mathtt{MERB5}$ methods:
\setlength\abovedisplayskip{6pt}
\begin{equation} \label{eq:MERB5}
\begin{aligned}
\hat{p}_{n2}(\tau) &=  \nnltermn + (t_n + \tau) \vn, \quad \hspace{3.6cm} \tau \in [0,c_2H]\\
\hat{p}_{n3}(\tau) &\equiv  \hat{p}_{n4}(\tau) = \nnltermn + (t_n + \tau) \vn +  \big(\tfrac{\tau}{c_2 H}\big)^2\Dntwo, \quad \tau \in [0,c_3H]\\
\hat{q}_n(\tau) &= \nnltermn + (t_n + \tau) \vn  + \tfrac{\tau^2}{H^2}\big(\tfrac{c_4}{c_3^2 (c_4 - c_3)} \widehat{D}_{n3} + \tfrac{c_3}{c_4^2(c_3 - c_4)}\widehat{D}_{n4} \big)  \\
& - \tfrac{\tau^3}{H^3} \big(\tfrac{1}{c_3^2(c_4 - c_3)}\widehat{D}_{n3} + \tfrac{1}{c_4^2(c_3-c_4)} \widehat{D}_{n4}\big), \quad \hspace{2.2cm}   \tau \in [0,H].
\end{aligned}
\setlength{\belowdisplayskip}{6pt}
\end{equation}
This only requires solving 3 modified ODEs per slow step (the only existing fifth-order multirate method, MERK5, requires 5).  In our experiments we choose $c_2 = c_4 = \frac{1}{4} < c_3 = \frac{33}{40}$, so we can solve the modified ODE \eqref{eq:MERBa} using the polynomial $\hat{p}_{n3}(\tau)$ on $[0, c_3 H]$ to obtain \emph{both} $\widehat{U}_{n3} \approx U_{n3}$ and $\widehat{U}_{n4} \approx U_{n4}$ (since $c_4<c_3$), without solving an additional fast ODE on $[0, c_4 H]$.
Using this strategy, the total fast traversal time for $\mathtt{MERB5}$ is $(1+c_2+c_3)H = \frac{83}{40}H=2.075H$.
\subsubsection{Sixth-order methods}
To the best of our knowledge, the only existing ExpRB method of order 6, named $\mathtt{pexprb65s7}$, is given in \cite{LO16}. It uses $s=7$ stages
and is embedded with a fifth-order method.  As with \eqref{eq:expRB5}, this method consists of multiple independent internal stages (namely the stages in two groups $\{U_{n2}, U_{n3}\}$ and $\{U_{n4}, U_{n5}, U_{n6}, U_{n7}\}$) that can be computed simultaneously, which we exploit to implement like a 3-stage method.  While $\mathtt{pexprb65s7}$ is constructed for autonomous problems and uses a set of fixed nodes $c_i$, we extend the derivation from \cite{LO16} to construct a family of 7-stage sixth-order methods
for non-autonomous problems:
 \begin{equation} \label{eq:expRB6}
\begin{aligned}
U_{nk}= u_n   &+ c_k H \varphi _{1} (c_k  H J_n)F(t_n, u_n) +  (c_k H)^2 \varphi _{2} (c_k  H J_n)V_n, \ k = 2, 3 \\
U_{n i} = u_n  &+ c_i H \varphi _{1} (c_i  H J_n)F(t_n, u_n) +  (c_i H)^2 \varphi _{2} (c_i  H J_n)V_n, \\
& + H a_{i2}(H J_n) D_{n2} + H a_{i3}(H J_n) D_{n3},  \ \hspace{1.7cm} i = 4, 5, 6, 7\\
u_{n+1} = u_n  & + H \varphi _{1} ( H J_n)F(t_n, u_n) + H^2 \varphi _{2} ( H J_n) V_n +H\sum_{i=4}^{7} b_i (H J_n) D_{ni},
\end{aligned}
\setlength{\belowdisplayskip}{2pt}
\end{equation}
where
\begin{align*}
 a_{i2}(H J_n) & = \tfrac{1}{c_2^2 (c_3 - c_2)} \big(2 c_i^3 c_3 \varphi_{3} (c_i H J_n) - 6 c_i^4 \varphi_{4} (c_i H J_n) \big), \\
 a_{i3}(H J_n) & = \tfrac{1}{c_3^2 (c_2 - c_3)} \big(2 c_i^3 c_2 \varphi_{3} (c_i H J_n) - 6 c_i^4 \varphi_{4} (c_i H J_n) \big),\\
  b_i (H J_n)  & = -2  \hat{\alpha}_i \varphi_{3} ( H J_n) + 6  \hat{\eta}_i \varphi_{4} ( H J_n) - 24  \hat{\beta}_i \varphi_{5} ( H J_n) + 120 \hat{\gamma}_i \varphi_{6} ( H J_n),\\
 \hat{\gamma}_i &= \dfrac{1}{c^2_i (c_i-c_k)(c_i - c_l)(c_i-c_m)}, \qquad
 \hat{\alpha}_i =c_k c_l c_m \hat{\gamma}_i, \\
 \hat{\beta}_i &=(c_k + c_l+c_m) \hat{\gamma}_i, \qquad
 \hat{\eta}_i  = (c_k c_l + c_l c_m + c_k c_m) \hat{\gamma}_i.
\end{align*}
Here  $i, k, l, m \in \{4, 5, 6, 7\}$ are distinct indices and $c_i, c_k, c_l, c_m$ are distinct (positive) nodes.  Applying Lemma~\ref{Lemma2.1} we obtain the first-ever sixth-order infinitesimal multirate method, $\mathtt{MERB6}$:
\begin{equation*} \label{eq:MERB6}
\begin{aligned}
\hat{p}_{n2}(\tau) &\equiv \hat{p}_{n3}(\tau)= \nnltermn + (t_n + \tau) \vn, & \tau &\in [0,c_2H]\\
\hat{p}_{n4}(\tau) & \equiv \hat{p}_{n5}(\tau) \equiv  \hat{p}_{n6}(\tau) \equiv  \hat{p}_{n7}(\tau)=
     \nnltermn  + (t_n + \tau) \vn \\
    &+  \tfrac{\tau^2}{(c_3 - c_2)H^2}\Big(\tfrac{c_3}{c_2^2 } \widehat{D}_{n2} - \tfrac{c_2}{c_3^2}\widehat{D}_{n3} \Big)
    - \tfrac{\tau^3}{(c_3 - c_2)H^3} \Big(\tfrac{1}{c_2^2}\widehat{D}_{n2} - \tfrac{1}{c_3^2} \widehat{D}_{n3}\Big), & \tau &\in [0,c_4H]\\
    \hat{q}_n(\tau) &= \nnltermn + (t_n + \tau) \vn
    - \tfrac{\tau^2}{H^2} \sum_{i=4}^{7} \hat{\alpha}_i \widehat{D}_{ni}
    + \tfrac{\tau^3}{H^3} \sum_{i=4}^{7} \hat{\eta}_i \widehat{D}_{ni}\\
   & \hspace{2cm}  - \tfrac{\tau^4}{H^4} \sum_{i=4}^{7} \hat{\beta}_i \widehat{D}_{ni}
    + \tfrac{\tau^5}{H^5} \sum_{i=4}^{7} \hat{\gamma}_i \widehat{D}_{ni}.
    & \tau &\in [0,H]
\end{aligned}
\end{equation*}
As seen, $\mathtt{MERB6}$ requires only 3 modified ODEs per each slow time step like $\mathtt{MERB5}$, reflecting the fact that its base 6th-order ExpRB method \eqref{eq:expRB6} has the structure of a 3-stage method. $\mathtt{MERB6}$ can be also implemented in an efficient way by choosing  $c_3 < c_2$ and $c_5, c_6, c_7 < c_4$. With these choices, we can solve the modified ODE \eqref{eq:MERBa} using $\hat{p}_{n2}(\tau)$ on $[0, c_2 H]$ to obtain both $\widehat{U}_{n2}
\approx U_{n2}$ and $\widehat{U}_{n3} \approx U_{n3}$ without solving an additional fast ODE on $[0, c_3 H]$. Similarly, we can solve \eqref{eq:MERBa} using $\hat{p}_{n4}(\tau)$ on $[0, c_4 H]$ to get \emph{all four approximations} $\widehat{U}_{ni}
\approx U_{ni}$ ($i = 4, 5, 6, 7$) without solving 3 additional ODEs on $[0, c_5 H]$, $[0, c_6 H]$, and $[0, c_7 H]$.
In our numerical experiments, we take $c_3= c_5 = \frac{1}{10} < c_2 =c_6= \frac{1}{9}<c_7=\frac{1}{8}<c_4=\frac{1}{7}$. This gives a total fast traversal time of $(1+c_2+c_4)H=\frac{79}{63}H\approx1.253H$.

\subsection{MERB method implementation}
\label{section:2.5}
In Algorithm~\ref{merbalgo}  we provide a precise description of the MERB algorithm.
\begin{algorithm}[h!]
\caption{MERB method}
\label{merbalgo}
\begin{list}{$\bullet $}{}
\item \textbf{Input:}  $F$; $J$; $V$; $t_0$; $u_0$; $s$; $c_i$ ($i=1,\ldots,s$); $H$
\item \textbf{Initialization:}  Set $n=0$; $\hat{u}_n=u_0$.\\
While $t_n<T$
\begin{enumerate}
  \item Set $\widehat{U}_{n1}=\hat{u}_n$.
  \item Compute $\widehat{J}_n = J(t_n,\hat{u}_n)$ and $\widehat{V}_n = V(t_n,\hat{u}_n)$
\item For $i=2,\ldots,s$ do
\begin{enumerate}
  \item Find  $\hat{p}_{ni}(\tau)$ as in \eqref{eq:MERB:poly:p}.
  \item Solve \eqref{eq:MERBa} on $[0, c_i H]$ to obtain $\widehat{U}_{ni}\approx y_{ni}(c_i H)$.
\end{enumerate}
 \item Find $\hat{q}_{n}(\tau)$ as in \eqref{eq:MERB:poly:q}
  \item Solve \eqref{eq:MERBb} on $[0, H]$ to get $\hat{u}_{n+1}\approx y_{n+1}(H).$
  \item Update $t_{n+1}:=t_n+H$, $n:=n+1$.
\end{enumerate}
\item \textbf{Output:} Approximate values $\hat{u}_n\approx u_n, n=1,2,\ldots$ (where
$u_n$ is the numerical solution at time $t_n$ obtained by an ExpRB method).
\end{list}
\end{algorithm}
We note that in our implementations of MERB methods, we found it beneficial to include formulas for $\widehat{N}_n(t,u)$ and $\widehat{D}_{ni}(t,u)$ as additional inputs to the algorithm (provided they can be pre-computed) for use in equations \eqref{eq:MERB:poly:p} and \eqref{eq:MERB:poly:q} to avoid floating-point cancellation errors when seeking very accurate solutions.  On the other hand, we note that within the MERB algorithm, both the products $Jw$ and $V\tau$ can be approximated from $F$ using finite differences,
\begin{align*}
   J(t,u) w &= \tfrac{1}{\sigma} \left(F(t,u+\sigma w)-F(t,u)\right) + \mathcal{O}(\sigma),\q\text{and}\\
   V(t,u) \tau &= \tfrac{1}{\sigma} \left(F(t+\sigma \tau,u)-F(t,u)\right) + \mathcal{O}(\sigma),
\end{align*}
instead of $J$ and $V$ being provided analytically; however, when seeking high accuracy then such approximations can cause excessive floating-point cancellation error.


\section{Numerical Experiments} \label{sec:numerical_results}

In this section, we implement MERB methods on select multirate test problems to demonstrate accuracy and efficiency.
We first discuss choices for the inner fast integrators, fast-slow splitting, optimal time scale separation factors, and give a general description of how the error and efficiency are measured.
We then present numerical results for a reaction-diffusion problem and a semi-linear nonautonomous system with coupling between the fast and slow variables.
For each problem, we compare the proposed \merbthree, \merbfour, \merbfive, and \merbsix~ methods with other recently developed multirate methods that treat the slow time scale explicitly, namely \merkthree, \merkfour, and \merkfive~ from \cite{LCR2020}, plus \mrigarkthree~ and \mrigarkfour~ from \cite{Sandu2019}, written here in short form as \mrigarkth~ and \mrigarkf. MATLAB implementations of all tests are provided on Github \cite{MERB_repo}.

\subsection{Choice of inner integrators}

\par
For uniformity in our implementations of MERB, MERK, and MRI-GARK methods of the same order, we use the same explicit fast integrators for solving all modified ODEs.  Third-order methods use a 3-stage explicit third-order method from equation (233f) of \cite{butcherNumericalMethodsOrdinary2008}, fourth-order methods use a 4-stage explicit fourth-order method commonly known as ``RK4" from \cite{kutta1901}, fifth-order methods use an 8-stage fifth-order method which is the explicit part of ARK5(4)8L[2]SA from \cite{kennedyAdditiveRungeKutta2003}, while the sixth-order method uses an 8-stage explicit sixth-order method based on the 8,5(6) procedure of \cite{vernerExplicitRungeKutta1978}. We note that although both MERK and MERB methods could compute the internal stages using a lower order integrator, for the sake of simplicity that approach is not used here.

\subsection{Fast-slow splitting}

\par The splitting of an IVP into fast and slow components, $u'(t) = F(t,u) = F_f(t,u) + F_s(t,u)$, for MERB methods is dictated by the dynamic linearization process at each time step,
\begin{equation}
\label{eq:dynamic_linearization}
  \hat{u}'(t)=F(t,\hat{u}(t)) = \hat{J}_{n} \hat{u}(t)+ \hat{V}_n t+ \hat{N}_n(t,\hat{u}(t)),
\end{equation}
where the multirate splitting becomes $F_f(t,u) = \hat{J}_n u$ and $F_s(t,u) = \hat{V}_n t+ \hat{N}_n(t,u)$.  This brings interesting questions when comparing against MERK and MRI-GARK methods that do not require dynamic linearization.  MERK methods require that $F_f(t,u) = \mathcal{L}u$, but MRI-GARK methods have no constraints on $F_f$ or $F_s$.  Thus to provide a more thorough picture in the following comparisons of MERB, MERK and MRI-GARK methods, we consider two separate fast-slow splittings for each problem.  The first is the dynamic linearization \eqref{eq:dynamic_linearization}, that can place more of a problem's dynamics at the fast time scale than other fixed multirate splittings; this offers a potential for greater multirate accuracy but at the expense constructing the dynamic linearization at each slow step.
Our second splitting defines a fixed $F_f(t,u) = \mathcal{L}u$, leaving $F_s(t,u) = F(t,u) - \mathcal{L}u$; in the ensuing results we call this the `fixed linearization'.
Though the motivation for this splitting arises from the MERK requirement on $F_f$, we also apply this splitting to MRI-GARK methods.
We note that other fixed splittings which can offer different accuracy and efficiency insights on multirate methods are possible, however we only focus on one fixed splitting for each test problem.
Methods that use fixed linearization are denoted with an asterisk in our results, for instance, \merkthrees~ uses a fixed linearization while \merkthree~ uses dynamic linearization.

\subsection{Optimal time scale separation}

In order to compare methods at their peak performance, we strive to determine an optimal time scale separation factor for each multirate method on each test problem. The optimal time scale separation factor $m = H/h$ is the integer ratio between the slow and fast time step sizes that results in maximal efficiency. We follow the approach from \cite{LCR2020} for determining this value experimentally, by comparing efficiency in terms of slow-only function evaluations and total (slow+fast) function evaluations for several different values of $m$ and $H$. 

\begin{table}[h!]
\centering
\begin{tabular}{| *{8}{c|}}
    \hline
        \thead{Method} & \thead{Slow \\stages}& \thead{Modified \\ODEs} & \thead{Fast time\\ traversal\\ of $[0,H]$} & \multicolumn{2}{c|}{\thead{React.-diffusion \\ optimal $m$}} & \multicolumn{2}{c|}{\thead{Bidirect.~coupling\\ optimal $m$}} \\

       & & & & \thead{Dynamic} & \thead{Fixed} & \thead{Dynamic} & \thead{Fixed} \\
        \hline
        \merbthree & $2$ & $2$ & $1.5$  &$10$ &  &  $80$ &  \\
        \hline
        \merkthree & $3$ & $3$ & $2.166$ &$20$ & $10$ & $80$ &  $10$\\
        \hline
        \mrigarkth & $3$ & $3$ &$1$             &$20$ & $5$ & $80$ & $10$\\
        \hline
        \merbfour  & $2$ & $2$ &$1.75$   &$10$ &  &  $40$ & \\
        \hline
        \merkfour  & $6$ & $4$ & $2.833$  &$20$ & $10$ & $40$  &$10$  \\
        \hline
        \mrigarkf  & $5$ & $5$ & $1$              &$10$  &$1$  & $40$ & $1$ \\
        \hline
        \merbfive  & $4$ & $3$ & $2.075$  &$5$ &   & $10$ & \\
        \hline
        \merkfive  & $10$& $5$ &$3.2$   &$5$ & $5$& $10$ & $10$ \\
        \hline
        \merbsix   & $7$ & $3$ &$1.253$ &$5$ &  & $5$ & \\
        \hline
\end{tabular}
\caption{Multirate method properties: number of slow internal stages and modified ODEs, total step traversal times, and optimal $m$ factors for each problem and splitting.}
\label{table:optm}
\end{table}

\par Table~\ref{table:optm} presents the optimal $m$ values for each method and each test problem splitting. A trend emerges among MERK and MRI-GARK methods that use both dynamic and fixed linearization: dynamic linearization almost exclusively results in larger optimal $m$ values than fixed linearization, supporting our earlier hypothesis that dynamic linearization includes more of the problem within the fast dynamics, thereby requiring a larger $m$ value to resolve. We also note that for the fixed linearization, both MRI-GARK methods have smaller $m$ values than other methods, resulting in less computational work at the fast time scale for a given $H$ value.

\subsection{Presentation of results}
For each test problem we sort our results into 3 groups: $\mathcal{O}(H^3)$ methods, $\mathcal{O}(H^4)$ methods, and $\mathcal{O}(H^5)$ with $\mathcal{O}(H^6)$ methods.
In each group we provide four kinds of ``log-log" plots: one convergence plot (error versus $H$) and three efficiency plots that measure cost through slow function calls, total function calls, and MATLAB runtimes, respectively.
Solution error is computed as the maximum absolute error over all spatial grid points and time outputs, as measured against either an analytical solution or highly accurate reference solution.
We also compute convergence rates using a linear least squares fit of the $\log($error$)$ versus $\log(H)$ data, neglecting points at the reference solution floor. Each of our efficiency measurements tells a different story. First, slow function calls illustrate the cost of a multirate method when applied to IVP systems with expense dominated by the slow components $F_s(t,u)$.  Second, total function calls capture the cost of $F_f(t,u)$, and highlight properties of methods related to their total fast traversal times.  Lastly, even though MATLAB runtimes are a poor proxy for runtimes on HPC applications, we use them here to capture the costs associated with dynamic linearization, and to measure how these costs affect efficiency.

\subsection{Reaction-diffusion} \label{subsec:reacdiff}
From Savcenco et al.\cite{savcenco}, we consider the reaction-diffusion equation:
\begin{equation}
    u_t = \epsilon u_{xx} + \gamma u^2 (1-u), \quad 0<x<5,  \quad 0<t\leq 5.
\end{equation}
 The initial and boundary conditions are $u(x,0) = (1+\exp(\lambda (x-1)))^{-1}$ and $u_x(0,t) = u_x(5,t) = 0$ respectively, where $\lambda = \frac{1}{2}\sqrt{2\gamma/\epsilon}$. Multiple combinations of $\gamma$ and $\epsilon$ are possible, here we choose $\gamma = 0.1$ and $\epsilon = 0.01$ that lead to an optimal $m >1$ when using dynamic linearization. We use a second-order centered finite difference scheme with $101$ spatial grid points to discretize the diffusion term.
In addition to dynamic linearization, MERK and MRI-GARK methods also use a fixed splitting where $F_{f}(t,u) = \epsilon u_{xx}$ and $F_{s}(t,u) = \gamma u^2 (1-u)$.
The numerical solution is considered at 10 evenly spaced points within the time interval, and all methods are tested with slow time steps $H = 0.5\times 2^{-k}$, for $k=0,\ldots,6$.
We compute error by comparing against a reference solution obtained using MATLAB's \texttt{ode15s} with relative tolerance $10^{-13}$ and absolute tolerance $10^{-14}$.

\begin{figure}[htbp]
   \centering
\includegraphics[width=0.9\textwidth]{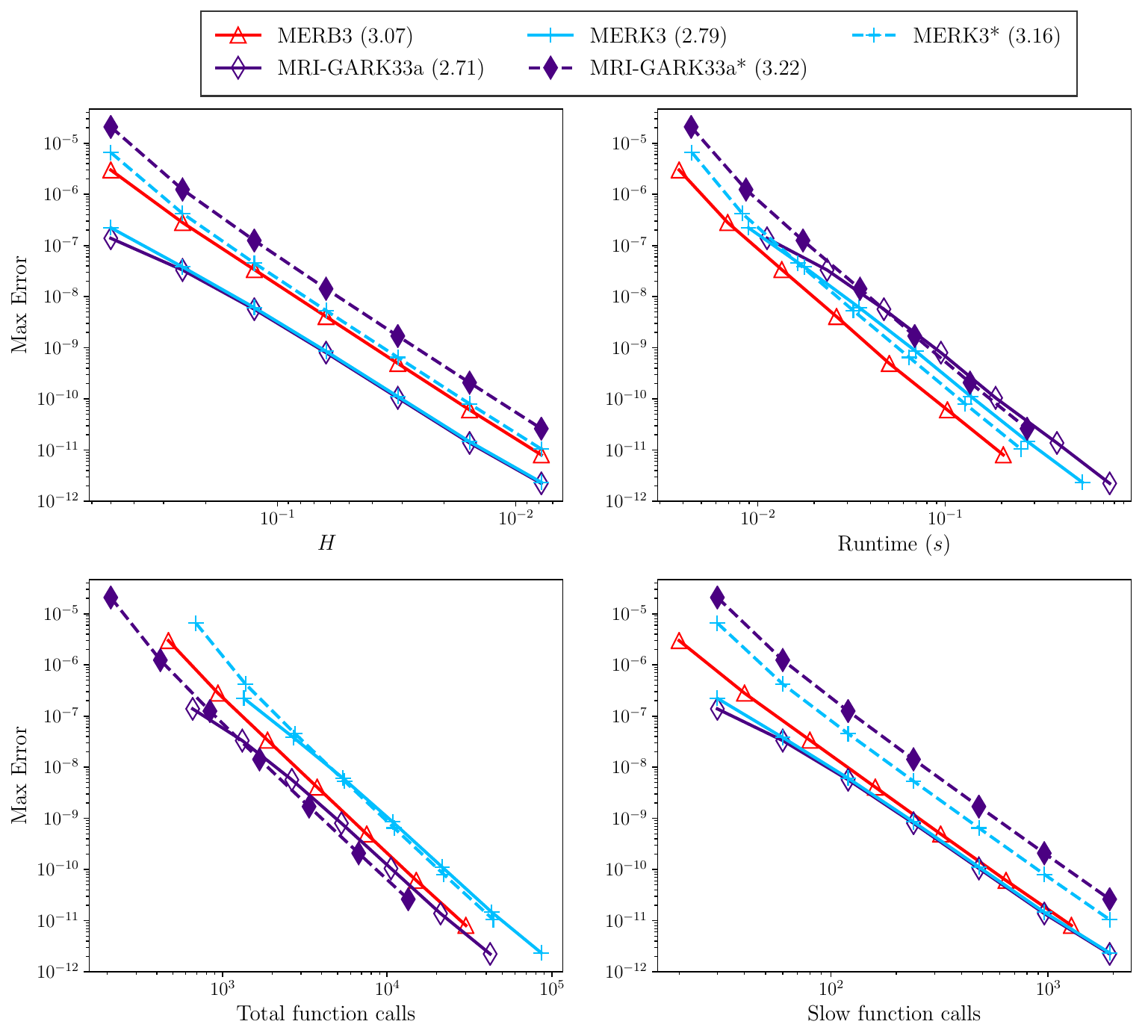}
   \caption{Convergence (top-left) and efficiency (top-right, bottom) for $\mathcal{O}(H^3)$ methods on the reaction-diffusion problem of Section~\ref{subsec:reacdiff}.
   The legend displays the measured convergence rates for each method in parentheses.
   }
   \label{fig:randdthirdordereff}
\end{figure}

Figures~\ref{fig:randdthirdordereff}-\ref{fig:randdfifthordereff} show accuracy and efficiency results for this problem. Examining the top-left of  Figure~\ref{fig:randdthirdordereff} and the legend, we see that each third-order method attains the expected order of convergence. The observed errors for the dynamic linearization approach on all methods are less than for fixed linearization. This can be attributed to inclusion of more of the problem at the fast time scale in the case of dynamic linearization, which results in higher optimal time scale separation factors (as shown in Table~\ref{table:optm}) and lower errors. Among the methods that apply dynamic linearization, \merkthree~ and \mrigarkth~ have almost identical errors that are lower than those for \merbthree, which uses an $m$ two times smaller. \merkthrees~ and \mrigarkths~ have the largest errors on this test problem.

Turning to the efficiency results at the top-right and bottom of Figure~\ref{fig:randdthirdordereff}, the most efficient methods in each of these plots are closest to the bottom left corner. For our MATLAB implementations, \merbthree~ has an obvious advantage in terms of runtime, while both \mrigarkth~ and \mrigarkths~ have the least efficient implementation.  Taking into account only MERK and MRI-GARK methods, there is not a significant runtime difference between the dynamic fixed linearization approaches, although the fixed linearization is very slightly faster at tighter accuracies.
When looking at total function calls, both \mrigarkth~ and \mrigarkths~ are the most efficient of the group, largely owing to their shorter fast traversal time of $1.0H$, while \merbthree~ and \merkthree~ have traversal times of $1.5H$ and $2.166H$, respectively.
The slow function call efficiency is closely aligned with the convergence behavior: at large values of $H$, \merkthree~ and \mrigarkth~ are the most efficient, but \merbthree~ is just as efficient as \merkthree~ and \mrigarkth~ at tighter accuracies.

\begin{figure}[htbp]
    \centering
    \includegraphics[width=0.9\textwidth]{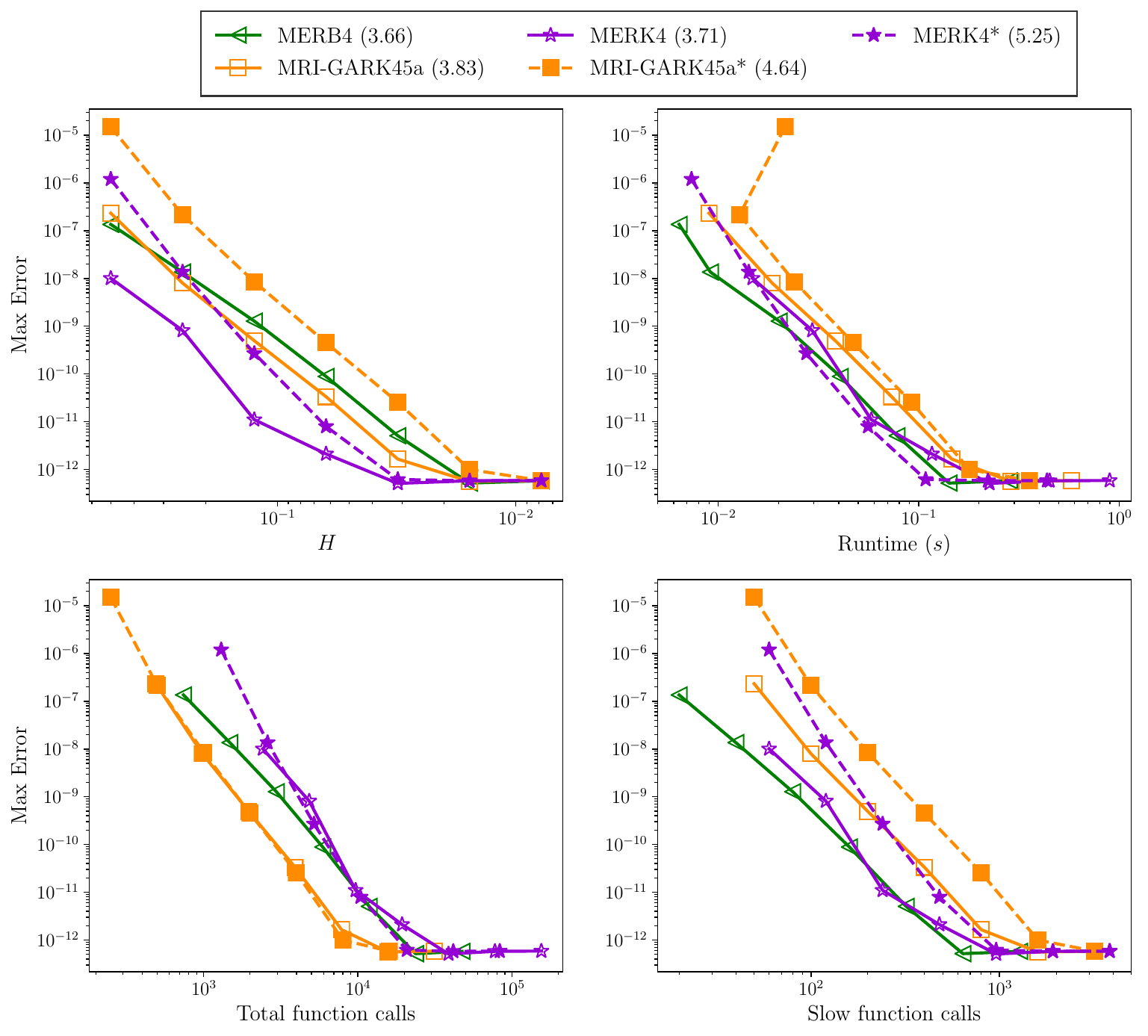}
   \caption{Convergence (top-left) and efficiency (top-right, bottom) for $\mathcal{O}(H^4)$ methods on the reaction-diffusion problem of Section~\ref{subsec:reacdiff}.
   }
   \label{fig:randdfourthordereff}
\end{figure}

In Figure~\ref{fig:randdfourthordereff} we present the corresponding plots for the fourth-order methods.  Here, all methods approximately reach their expected order of convergence, with \merkfours~ and \mrigarkfs~ outperforming their expectations.
\merkfour~ has the smallest error, but also uses an $m$ value that is two times greater than other fourth-order methods on this test problem (see Table~\ref{table:optm}). \merkfours~ starts off with larger errors than \merbfour~ and \mrigarkf, but because it converges at fifth-order for this test problem, its errors quickly drop below those for \merbfour~ and \mrigarkf. \mrigarkfs~ has an $m=1$ which seemingly puts it at a disadvantage when comparing accuracy with other methods, however, larger values of $m$ only lead to more total function evaluations with no reduction in error.  Focusing on runtime efficiency, \merbfour~ is more efficient at larger error values, but \merkfour~ is eventually the most efficient at smaller error values.  Total function call efficiency repeats the previous pattern from third-order methods: \mrigarkf~ and \mrigarkfs~ are the most efficient and closely line up, \merbfour~ performs better than \merkfour~ and \merkfours due to its shorter total traversal time of $1.75H$ versus $2.833H$.
Finally, when comparing slow function calls \merbfour~ is the most efficient. This is expected since \merbfour~ only has 2 slow stages, compared with 6 for \merkfour~ and 5 for \mrigarkf.

\begin{figure}[htbp]
   \centering
   \includegraphics[width=0.9\textwidth]{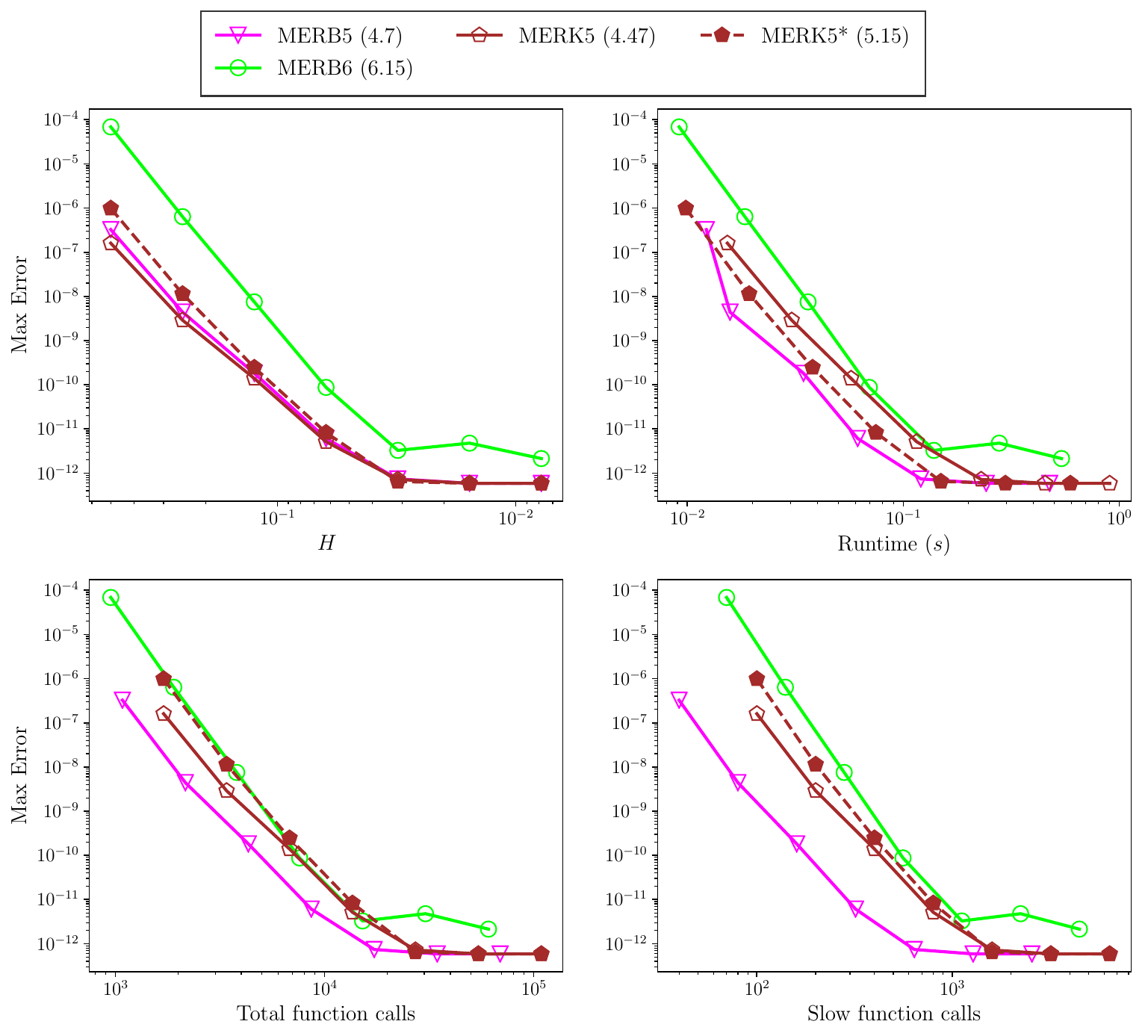}
   \caption{Convergence (top-left) and efficiency (top-right, bottom) of $\mathcal{O}(H^5)$ and $\mathcal{O}(H^6)$ methods on the reaction-diffusion problem of Section \ref{subsec:reacdiff}.
   }
    \label{fig:randdfifthordereff}
\end{figure}

The first thing to note discussing the fifth and sixth-order methods is that they all use the same $m=5$ for this problem (Table~\ref{table:optm}).  Their convergence and efficiency plots are provided in Figure~\ref{fig:randdfifthordereff}.  On this problem, all methods converge at their expected rates, although \merbsix~ starts out with larger error values than the fifth-order methods, that all cluster around similar error values, although the dynamic linearization used by \merkfive~ results in slightly less error than the fixed linearization used in \merkfives.  In all three measures of efficiency \merbfive~ is the most efficient. Looking at total function calls efficiency, \merbfive~ has a total traversal time of $2.075H$ compared to $3.2H$ for \merkfive~ and $1.253H$ for \merbsix~ (though we barely get to see advantages of this due to its larger error on this problem). When it comes to slow function calls, \merbfive's 4 slow stages is much lower than the 10 stages for \merkfive and 7 stages for \merbsix.  Combining the merits of \merbfive~ from total function calls and slow function calls explains its runtime efficiency performance.

\subsection{Bidirectional coupling system} \label{subsec:bcnonauto}
Inspired by \cite[Sect. 5.1]{estep}, we propose the following semi-linear, nonautonomous bidirectional coupling problem on $0<t\leq1$:
\begin{subequations} \label{eq:bicoupling}
 \begin{align}
    \label{eq:bicoupling_u}
    u' &= \sigma v - w - \beta t,\\
    \label{eq:bicoupling_v}
    v' &= -\sigma u,\\
    \label{eq:bicoupling_w}
    w' &= -\lambda(w + \beta t) - \beta \Bigg(u - \frac{a(w + \beta t)}{a\lambda + b\sigma} \Bigg)^2 - \beta \Bigg(v - \frac{b(w + \beta t)}{a \lambda + b\sigma}\Bigg)^2,
\end{align}
\end{subequations}
with exact solution $u(t) = \cos(\sigma t) + ae^{-\lambda t}, \quad v(t) = -\sin(\sigma t) + be^{-\lambda t},\quad \text{and} \quad w(t) = (a\lambda + b \sigma)e^{-\lambda t} - \beta t.$
This problem features linear coupling from slow to fast time scales through the equation \eqref{eq:bicoupling_u}, and nonlinear coupling from fast to slow time scales through the equation for \eqref{eq:bicoupling_w}.
In addition, it includes tunable parameters $\{a, b, \beta, \lambda, \sigma\}$ taken here to be $\{1, 20,0.01, 5, 100\}$, with $a\sigma = b\lambda$;  $\sigma$ determines the frequency of the fast time scale and $\beta$ controls the strength of the nonlinearity. In the case of dynamic linearization, smaller values of $\beta$ correspond with weaker nonlinearity, resulting in higher values of the optimal time scale separation factor $m$.
While there are various possible fixed splittings, we chose the most natural splitting into fast variables and slow variables informed by the exact solution:

{\small
\begin{align*}\label{eq:bidirectional_rhs}
F_{f}(t,\mathbf{u}) = \begin{bmatrix} \sigma v\\
    -\sigma u \\ 0 \end{bmatrix}, \quad
F_{s}(t,\mathbf{u}) = \begin{bmatrix}-w-\beta t  \\ 0 \\ -\lambda(w + \beta t) - \beta \left(u - \frac{a(w - \beta t)}{a\lambda + b\sigma} \right)^2 - \beta \left(v - \frac{b(w - \beta t)}{a \lambda + b\sigma}\right)^2  \end{bmatrix}
\end{align*}}

We assess error at 20 equally spaced points within the time interval and consider time steps $H=0.05 \times 2^{-k}$ for integers $k=0,1,\ldots,7$.

\begin{figure}[htbp]
    \centering
    \includegraphics[width=\textwidth]{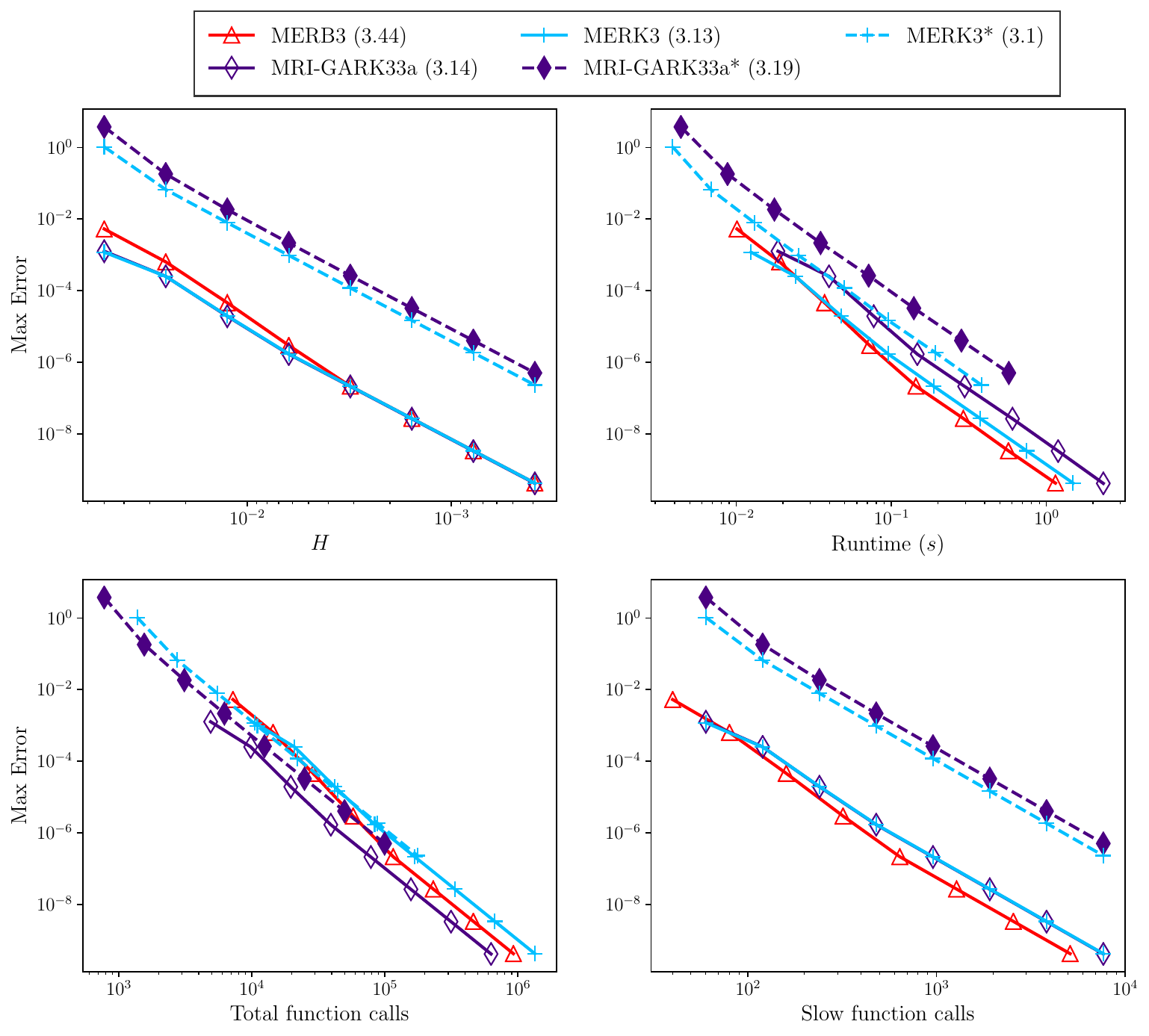}
   \caption{Convergence (top-left) and efficiency (top-right, bottom) of $\mathcal{O}(H^3)$ methods on the bidirectional coupling problem of Section \ref{subsec:bcnonauto}.
   }
   \label{fig:bcthirdordereff}
\end{figure}

\par Accuracy and efficiency plots for this problem are shown in Figures~\ref{fig:bcthirdordereff}-\ref{fig:bcfifthordereff}. Starting with third-order methods in Figure~\ref{fig:bcthirdordereff}, all methods incorporating dynamic linearization have similar errors, coinciding with their uniform time scale separation factor of $m=80$. Similarly, the methods using fixed linearization \merkthrees~ and \mrigarkths~ have the same $m=10$, leading to comparable errors. As before, dynamic linearization leads to lower errors than fixed linearization (here the difference in errors for the same $H$ is up to $10^3$).
The previous efficiency observations are repeated here as well: \merbthree~ is the most efficient in runtime and slow function evaluations, while \mrigarkth~ is slightly more efficient in total function evaluations. 

\begin{figure}[htbp]
    \centering
    \includegraphics[width=0.9\textwidth]{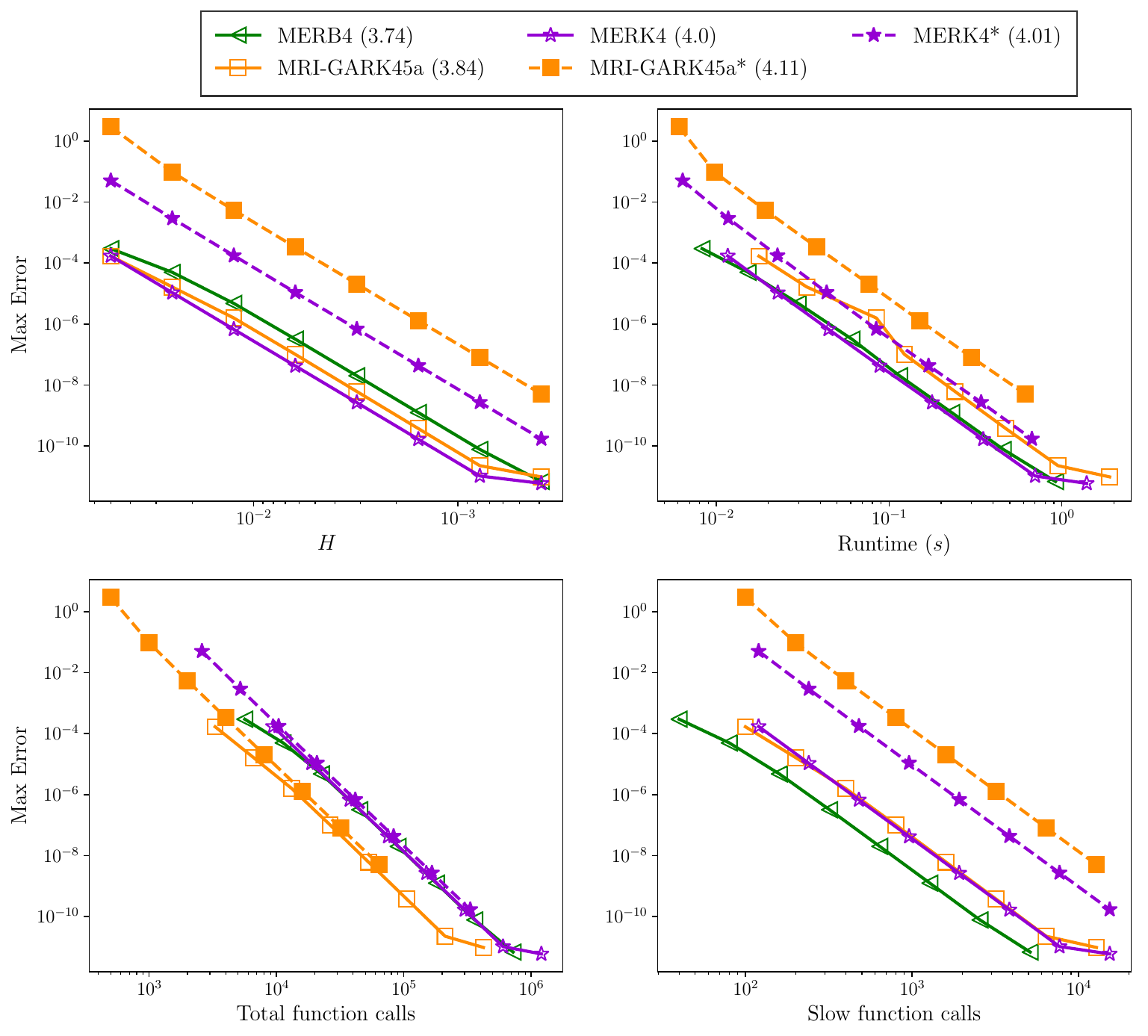}
   \caption{Convergence (top-left) and efficiency (top-right, bottom) of $\mathcal{O}(H^4)$ methods on the bidirectional coupling problem of Section \ref{subsec:bcnonauto}.
   }
\label{fig:bcfourthordereff}
\vspace{-5mm}
\end{figure}

\par Results for fourth-order methods are plotted in Figure~\ref{fig:bcfourthordereff}. Like with the third-order methods, we use the same $m$ for dynamic linearization methods, but here there is slightly more variation in errors, with \merkfour~ being slightly more accurate than the others in this group.  Both \merbfour~ and \merkfour~ show optimal runtime efficiency, the MRI-GARK methods are the most efficient in total function calls, and \merbfour~ is again the most efficient in slow function calls.

\begin{figure}[htbp]
    \centering
    \includegraphics[width=0.9\textwidth]{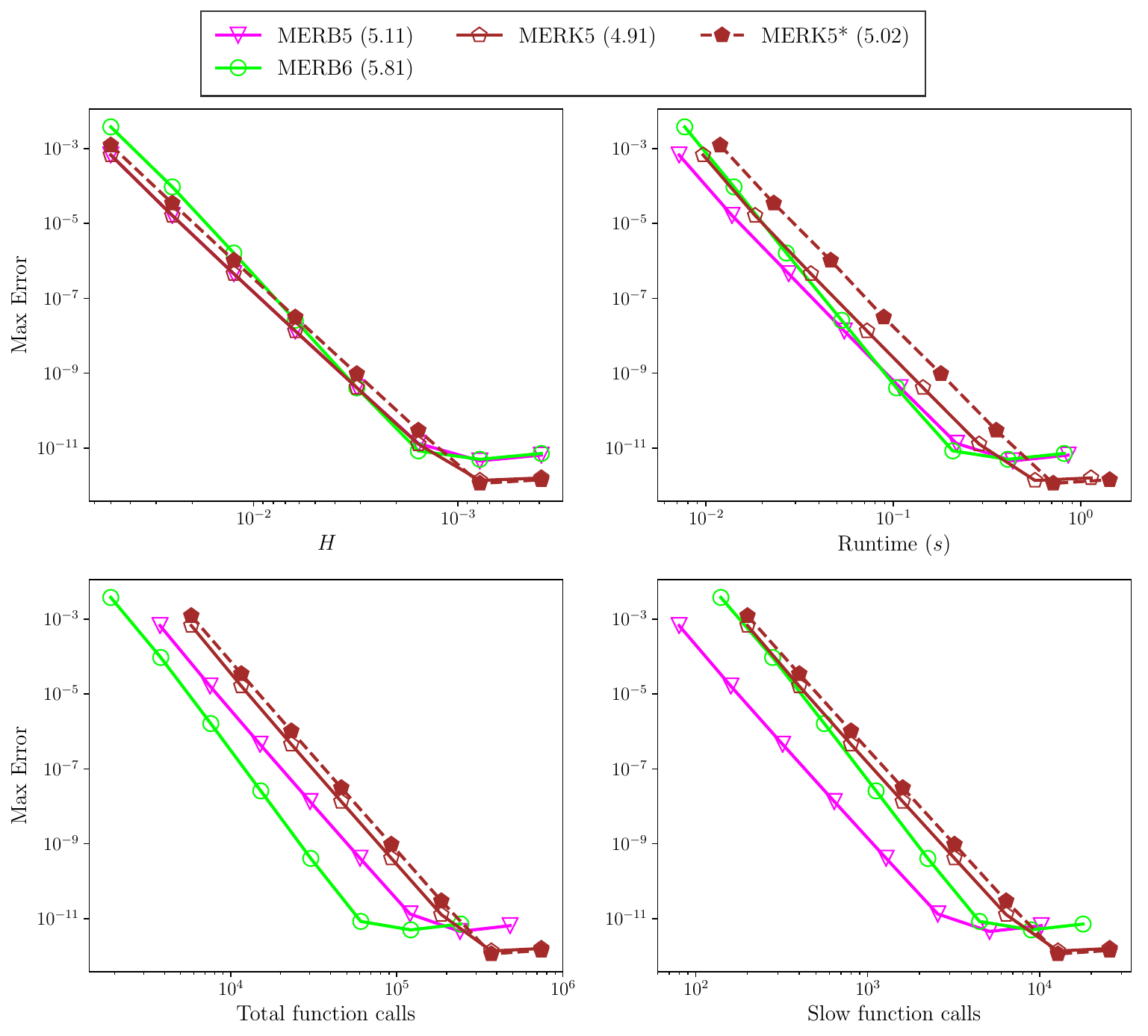}
   \caption{Convergence (top-left) and efficiency (top-right, bottom) of $\mathcal{O}(H^5)$ and $\mathcal{O}(H^6)$ methods on the bidirectional coupling problem of Section \ref{subsec:bcnonauto}.
   }
    \label{fig:bcfifthordereff}
\vspace{-5mm}
\end{figure}

\par Finally, the performance of fifth and sixth-order methods on the bidirectional coupling problem is illustrated in Figure~\ref{fig:bcfifthordereff}. The accuracy of these methods is almost identical on this test problem with \merbsix~ demonstrating a slightly steeper line, so we focus on the efficiency comparisons. Both of our new MERB methods are the most competitive for this test problem. We observe that \merbfive~ is slightly more efficient in terms of runtime at larger error values, but at smaller errors \merbsix~ becomes more efficient due to its higher order of accuracy. \merbsix~ is also the most efficient in total function calls followed by \merbfive, due to their smaller total traversal times in comparison with \merkfive.  The small number of stages for \merbfive~ makes it clearly more efficient in terms of slow function calls.

\section{Conclusions}
\label{sec:conclusion}

We have introduced a new approach for multirate integration of initial-value problems that evolve on multiple time scales.  Employing an MIS-like approach wherein the couplings between slow and fast time scales occurs through defining a sequence of modified IVPs at the fast time scale, and built off of existing ExpRB methods, the proposed MERB methods allow creation of multirate methods with very high order of accuracy, and minimize the amount of costly processing of the slow time scale operator.  In addition to deriving a clear mechanism for constructing these from certain classes of ExpRB schemes, we provide rigorous convergence analysis for MERB methods.  We note that the style of this analysis is much more elegant than our approach for MERK methods \cite{LCR2020}, in that we analyze the overall MERB error by separately quantifying the error between the MERB approximation of the underlying ExpRB method, and the error in the ExpRB approximation of the original IVP.  With this theory in hand, we propose a suite of MERB methods with orders 2 through 6, where in the cases of orders 3--6, we additionally provide generalizations of the base ExpRB methods and extend these to non-autonomous problems.

We examine the performance of the proposed MERB methods of orders 3 through 6, comparing these against existing MERK and explicit MRI-GARK methods on two test problems, and where the MERK and MRI-GARK methods are tested with two potential multirate splittings on each problem.  While all MERB, MERK and MRI-GARK methods exhibited their theoretical convergence rates on these problems and splittings, their efficiency differed considerably.  In order to provide results that potentially apply to a broad range of multirate applications, we investigate efficiency using three separate measurements of cost: MATLAB runtime, total function calls (both fast and slow), and slow function calls only.  Within these metrics, some general patterns emerge.  First, most of the methods exhibited optimal efficiency at higher $m=H/h$ values when using multirate splittings based on dynamic linearization as opposed to fixed splittings.  Second, the proposed MERB methods show the best runtime efficiency of all methods and splittings, although in some cases the equivalent order MERK method with dynamic splitting is competitive.  Third, due to their total fast time scale traversal times of $1.0H$, the MRI-GARK methods always exhibit the best total function call efficiency.  Lastly, due to their low number of slow stages, the proposed MERB methods are uniformly the most efficient when considering slow function calls (only in a few instances MERK with dynamic splitting was competitive). This is particular of interest for multirate problems where the fast component is much less costly to compute than the slow component.

Based on these results, we find that the newly proposed MERB methods provide a unique avenue to construction of high order MIS-like multirate methods, and that they are very competitive in comparison with other recently-developed high order MIS-like multirate schemes.  More work remains, however.  An obvious extension is to include embeddings to enable low-cost temporal error estimation, as well as to investigate robust techniques for error-based multirate time step adaptivity.  A further extension of MERB methods could focus on applications that require implicit or mixed implicit-explicit treatment of processes at the slow time scale.

\bibliographystyle{siamplain}
\bibliography{references}

\end{document}